\documentclass[sigconf]{acmart}

\usepackage{booktabs} 
\usepackage{algorithm}
\usepackage{algpseudocode}
\usepackage{amsmath,amssymb,amsfonts}
\usepackage{graphicx}
\usepackage{xcolor}
\usepackage{todonotes}
\usepackage{xspace}
\usepackage{multirow}
\usepackage{bm}
\usepackage{paralist}

\newtheorem{remark}{Remark}

\newcommand{\vx}{\mathbf{x}}
\newcommand{\vy}{\mathbf{y}}
\newcommand{\objfct}{\mathcal{L}}
\newcommand{\Xspace}{\mathcal{X}}
\newcommand{\Yspace}{\mathcal{Y}}
\newcommand{\horse}{{\textit{\sc Reckless}}\xspace}
\newcommand{\ES}{\textsf{ES}\xspace}
\newcommand{\EVOS}{Evolution Strategies\xspace}
\newcommand{\cmaes}{\texttt{CMA-ES}\xspace}
\newcommand{\nes}{\texttt{NES}\xspace}
\newcommand{\mmde}{\texttt{MMDE}\xspace}
\newcommand{\coevp}{\texttt{CoevP}\xspace}
\newcommand{\coeva}{\texttt{CoevA}\xspace}
\newcommand{\mse}{\texttt{MSE}}
\newcommand{\vmu}{\bm{\mu}}
\newcommand{\vepsilon}{\bm{\epsilon}}
\newcommand{\fes}{\text{\#FEs}\xspace}

\newcommand{\mx}{\mathbf{X}}
\newcommand{\my}{\mathbf{Y}}

\renewcommand{\eqref}[1]{Eq.~\ref{#1}}

\setcopyright{none}

\acmDOI{}

\acmISBN{}

\acmConference[]{February 2}
\acmYear{2018}
\copyrightyear{}

\acmPrice{}

\acmSubmissionID{}
\begin{document}
\title{On the Application of Danskin's Theorem to Derivative-Free Minimax Optimization}

\author{Abdullah Al-Dujaili}
\affiliation{%
  \institution{CSAIL, MIT}
  \city{Cambridge} 
  \state{USA} 
}
\email{aldujail@mit.edu}

\author{Shashank Srikant}
\affiliation{%
	\institution{CSAIL, MIT}
	\city{Cambridge} 
	\state{USA} 
}
\email{shash@mit.edu}

\author{Erik Hemberg}
\affiliation{%
	\institution{CSAIL, MIT}
	\city{Cambridge} 
	\state{USA} 
}
\email{hembergerik@csail.mit.edu}

\author{Una-May O'Reilly}
\affiliation{%
	\institution{CSAIL, MIT}
	\city{Cambridge} 
	\state{USA} 
}
\email{unamay@csail.mit.edu}


\begin{abstract}
	Motivated by Danskin's theorem, gradient-based methods have been applied with empirical success to solve minimax problems that involve non-convex outer minimization and non-concave inner maximization. On the other hand, recent work has demonstrated that Evolution Strategies (ES) algorithms are stochastic gradient approximators that seek robust solutions. 
	In this paper, we address black-box (gradient-free) minimax problems that have long been tackled in a coevolutionary setup. To this end and guaranteed by Danskin's theorem, we employ ES as a stochastic estimator for the descent direction.
	  The proposed approach is validated on a collection of black-box minimax problems. Based on our experiments, our method's performance is comparable with its coevolutionary counterparts and favorable for high-dimensional problems.  Its efficacy is demonstrated on a real-world application.
\end{abstract}

\maketitle

\section{Introduction}
\label{sec:intro}

Many real-world applications involve an adversary and/or uncertainty, specifically in the security domain. Consequently, several methods have been proposed to find solutions that have the best worst-case (or average) performance for security-critical systems. Important examples include face recognition~\cite{sharif2017adversarial} and malware detection~\cite{huang2018adversarial}.

The notion of security and adversarial robustness can be described by a minimax formulation~\cite{huang2015learning, madry2017towards}. The formulation is motivated by theoretical guarantees from Danskin's theorem~\cite{danskin1966theory} on using first-order information, i.e. gradients, to find or approximate solutions. Further, where theoretical guarantees can not be assumed, empirical solutions to problems, e.g. digit recognition, have been demonstrated~\cite{madry2017towards}.

In this paper, our interest is in black-box (gradient-free) minimax problems where, in contrast to the aforementioned examples of image recognition and malware, gradients are neither symbolically nor numerically available, or they are complex to compute~\cite{conn2009introduction}. This has led to extensive use of coevolutionary frameworks~\cite{herrmann1999genetic,jensen2001robust} to solve such problems. These frameworks however do not reconcile the guarantees provided by gradient-based frameworks in solving the minimax problem. Our goal is to bridge this divide and develop a method for black-box minimax that is consistent with the theoretical assumptions and guarantees of Danskin's theorem while using a gradient estimator in lieu of a gradient. For gradient estimation, we propose to employ a class of black-box optimization algorithms, viz. \EVOS (\ES).

 
Our proposition is motivated by the growing body of work~\cite{salimans2017evolution,morse2016simple} which has shown that the performance of gradient-based methods can be rivaled by \ES, and that \ES is more than just a traditional finite difference approximator~\cite{lehman2017more}. For more empirical and theoretical insights on \ES vs. gradient-based methods, see~\cite{JMLR:v18:14-467,wierstra2014natural,akimoto2012analysis}.

We report the following contributions:
\begin{inparaenum}[\itshape 1)]
\item We present a formal theoretical motivation for using \ES to meet the theoretical guarantees of using gradients for gradient-free minimax problems.
\item We present the \horse framework, which employs \ES to solve black-box minimax problems that have long been addressed in a coevolutionary setup.\footnote{\horse looks for minimax solutions and saddle points (if they exist), hence its name: \url{https://en.wikipedia.org/wiki/Sergeant_Reckless}}
\item We compare variants of \ES that have different properties in terms of their gradient estimation on a collection of benchmark problems.  Coevolutionary approaches have often been evaluated based on the Euclidean error of their solutions with respect to the optimal solutions. Our comparison is based on the notion of \emph{regret} (loss) in the objective function value at a solution with respect to that at an optimal solution.
\item We validate the effectiveness of the proposed framework and compare its performance in terms of regret against the state of the art coevolutionary algorithms on benchmark problems as well as a real-world application. We find that \horse scales better than the majority of its coevolutionary counterparts.
\item Finally, we provide the \horse framework and experiment code for public use.\footnote{The link to the code repository will be made available at \url{https://github.com/ALFA-group}.}
\end{inparaenum}

\section{Background}
\label{sec:background}
This section introduces the notation and terminology used in the rest of the paper, followed by a summary of coevolutionary approaches for black-box minimax problems.

\subsection{Black-Box Minimax Problem}
\label{sec:minmax-problem}
Formally, we are concerned with the black-box minimax optimization problem given a finite budget of function evaluations. Mathematically, the problem is a composition of an inner maximization problem and an outer minimization problem,
\begin{equation}
\min_{\vx \in \Xspace}\max_{\vy \in \Yspace} \objfct(\vx, \vy)\;,
\label{eq:minmax}
\end{equation}
where $\Xspace\subseteq \mathbb{R}^{n_x}$, $\Yspace\subseteq \mathbb{R}^{n_y}$, and $\objfct:\Xspace \times \Yspace \to \mathbb{R}$. The problem is called black-box because there is no closed-form expression of $\objfct$. Instead, one can query an oracle (e.g., a simulation) for the value of $\objfct$ at a specific pair $(\vx,\vy) \in \Xspace\times \Yspace$. The task is then to find the optimal solution $\vx^* \in \Xspace$ to~\eqref{eq:minmax}, whose corresponding objective function value $\objfct(\vx^*,\cdot)$ is at its max at $\vy^*\in\Yspace$, or a good approximate using a finite number of function evaluations, which are expensive in terms of computational resources (e.g. CPU time). 

The pair $(\hat{\vx}, \hat{\vy})\in \Xspace \times \Yspace$ is called a saddle point of~\eqref{eq:minmax} if $\forall \vx \in \Xspace$, $\forall \vy \in \Yspace$,
\begin{equation}
\objfct(\hat{\vx}, \vy) \leq \objfct(\hat{\vx}, \hat{\vy}) \leq \objfct(\vx, \hat{\vy})\;.
\label{eq:saddle-pt-def}
\end{equation}
Equivalently~\cite{jensen2001robust}, such pair exists if:
\begin{equation}
\min_{\vx \in \Xspace}\max_{\vy \in \Yspace} \objfct(\vx, \vy) = \max_{\vy \in \Yspace}\min_{\vx \in \Xspace} \objfct(\vx, \vy)\;.
\label{def:saddle-pt-symmetry}
\end{equation}
If a saddle point exists, then it follows from~\eqref{eq:saddle-pt-def}, that 
\begin{equation}
\min_{\vx \in \Xspace}\max_{\vy \in \Yspace} \objfct(\vx, \vy) =
max_{\vy \in \Xspace}\objfct(\hat{\vx}, \vy) =
min_{\vx \in \Xspace}\objfct(\vx, \hat{\vy})\;.
\end{equation}
From a game theory perspective, a saddle point represents a two-player zero-sum game equilibrium. Minimax problems with saddle points are referred to as symmetrical problems~\cite{jensen2003new}, in contrast to asymmetrical problems for which the condition~\eqref{eq:saddle-pt-def} does not hold. 

Given an algorithm's best solution $\vx_*$ to \eqref{eq:minmax}, the mean square Euclidean error (\mse) has been used as a performance metric to compare $\vx_*$ to the optimal solution $\vx^*$~\cite{qiu2017new,cramer2009evolutionary}. That is, 
\begin{equation}
\text{MSE}(\vx_*) = \frac{1}{n_x}||\vx_* - \vx^* ||^2_2\;.
\label{eq:mse}
\end{equation}
In this paper, we introduce a metric that is closely related to the notion of loss in decision-making problems under adversary and/or uncertainty as well as continuous optimization~\cite{bubeck2012regret,valko2013stochastic}: the \emph{regret} of an algorithm's best solution $\vx_*$ in comparison to the optimal solution $\vx^*$ is defined as
\begin{equation}
r(\vx_*) = \max_{\vy \in \Yspace} \objfct(\vx_*, \vy) - \objfct(\vx^*, \vy^*)\;,
\label{eq:regret}
\end{equation}
where the first term can be computed using an ensemble of off-the-shelf black-box continuous optimization solvers. Note that the regret measure, in comparison to \mse, allows us to compare the quality of several proposed solutions without the knowledge about the pair $(\vx^*,\vy^*)$ by just computing the first term of~\eqref{eq:regret}.

 \begin{algorithm}[h]
	\small
	\caption{Coevolutionary Algorithm Alternating~(\coeva) \newline
		\textbf{Input:}\newline 
		$~~$\hspace{\algorithmicindent}$T$: number of iterations,~$~~$\hspace{\algorithmicindent}$\tau$ : tournament size, \newline
		$~~$\hspace{\algorithmicindent}$\mu$: mutation probability,~$~~$\hspace{\algorithmicindent}$\lambda$ : population size
	}
	\label{alg:coevolutionary_algorithm_alternating}
	\begin{algorithmic}[1]
		\State $\vx_{1,0}, \ldots, \vx_{\lambda,0} \sim \mathcal{U}(\Xspace)$ 
		\State  $\vy_{1,0}, \ldots, \vy_{\lambda,0} \sim \mathcal{U}(\Yspace)$ 
		\State $\mx_0 \gets [\vx_{1,0}, \ldots, \vx_{\lambda,0}]$\Comment{Initialize minimizer population}
		\State $\my_0 \gets [\vy_{1,0}, \ldots, \vy_{\lambda,0}]$\Comment{Initialize maximizer population}
        \State $t\gets 0$
        \Repeat
        \State \texttt{sort}($\mx_{t},\objfct$)
        \State \texttt{sort}($\my_{t},\objfct$)
		\State $t \gets t + 1$ \Comment{Increase counter}
		\State $\mx_t \gets $ select from $s(\mx_{t-1}, \tau)$ \Comment{Tournament selection}
		\State $\mx_t \gets $ perturb values in $m(\mx_t, \mu)$ \Comment{Gaussian mutation}
		\State $\vx'_*, \vy'_* \gets \arg\min_{\vx \in \mx_t} \arg\max_{\vy \in \my_{t-1}} \objfct(\vx, \vy)$ \Comment{Best minimizer}
		\If {$\objfct(\vx'_*, \vy'_*) < \objfct(\vx_{\lambda, t-1}, \vy_{\lambda, t-1})$} \Comment{Replace worst minimizer}
		\State $\vx_{\lambda,t-1}\gets \vx'_*$ \Comment{Update population}
		\EndIf
		\State $\mx_t \gets \mx_{t-1}$ \Comment{Replicate population}
		\State $t \gets t + 1$ \Comment{Increase counter before alternating}
		\State $\my_t \gets $ select from $s(\my_{t-1}, \tau)$ \Comment{Tournament selection}
		\State $\my_t \gets $ perturb values with $m(\my_t, \mu)$ \Comment{Gaussian mutation}
		\State $\vx'_{0}, \vy'_{0} \gets \arg\min_{\vx \in \mx_{t}} \arg\max_{\vy \in \my_t} \objfct(\vx, \vy)$ \Comment{Best maximizer}
		\If {$\objfct(\vx'_{0}, \vy'_{0}) > \objfct(\vx_{\lambda,t}, \vy_{\lambda, t-1})$} \Comment{Replace worst maximizer}
		\State $\vy_{\lambda,t-1}\gets \vy'_{0}$  \Comment{Update population}
		\EndIf
		\State $\my_t \gets \my_{t-1}$ \Comment{Replicate population}
		\Until $t \geq T$
		\State $\vx_*, \vy_* \gets \arg\min_{\vx \in \mx_T} \arg\max_{\vy \in \my_T} \objfct(\vx, \vy)$ \Comment{Best minimizer}
		\State \Return $\vx_*, \vy_*$
	\end{algorithmic}
\end{algorithm}

\subsection{Coevolutionary Approaches}
\label{sec:coevo}

Coevolutionary algorithms for minmax problems maintain two populations, the first respresents \textit{solutions} and the second contains \textit{tests}\cite{herrmann1999genetic}. The fitness of \textit{solution} is determined by its performance when it interacts with some set of \textit{tests}. In this work, we explore two variants of this approach - \textit{alternating} and \textit{parallel}. In \textit{alternating}, the two populations (one for finding the minimum, and one for the maximum) take turns in evolving, learning from the previous iteration's population. In \textit{parallel}, the updates to both the populations happen in each step. Ideally there is convergence to a robust solution and its worst-case tests. The algorithms for these two approaches are listed in Algorithm~~\ref{alg:coevolutionary_algorithm_alternating} and \ref{alg:coevolutionary_algorithm_parallel}, see~\cite{back2000evolutionary} for selection and mutation operator details.

 \begin{algorithm}
    \small
    \caption{Coevolutionary Algorithm Parallel~(\coevp)  \newline
        \textbf{Input:}\newline 
        $~~$\hspace{\algorithmicindent}$T$: number of iterations,~$~~$\hspace{\algorithmicindent}$\tau$ : tournament size, \newline
        $~~$\hspace{\algorithmicindent}$\mu$: mutation probability,~$~~$\hspace{\algorithmicindent}$\lambda$ : population size
    }
    \label{alg:coevolutionary_algorithm_parallel}
    \begin{algorithmic}[1]
        \State $\vx_{1,0}, \ldots, \vx_{\lambda,0} \sim \mathcal{U}(\Xspace)$ 
        \State  $\vy_{1,0}, \ldots, \vy_{\lambda,0} \sim \mathcal{U}(\Yspace)$ 
        \State $\mx_0 \gets [\vx_{1,0}, \ldots, \vx_{\lambda,0}]$\Comment{Initialize minimizer population}
        \State $\my_0 \gets [\vy_{1,0}, \ldots, \vy_{\lambda,0}]$\Comment{Initialize maximizer population}
        \State $t\gets 0$
        \Repeat
        \State \texttt{sort}($\mx_{t},\objfct$)
        \State \texttt{sort}($\my_{t},\objfct$)
        \State $t \gets t + 1$ \Comment{Increase counter}
        \State $\mx \gets $ select from $s(\mx_{t-1}, \tau)$ \Comment{Tournament selection}
        \State $\my \gets $ select from $s(\my_{t-1}, \tau)$ \Comment{Tournament selection}
        \State $\mx \gets $ perturb values with $m(\mx,\mu)$ \Comment{Gaussian mutation}
        \State $\my \gets $ perturb values with $m(\my, \mu)$ \Comment{Gaussian mutation}
        \State $\vx'_*, \vy'_* \gets \arg\min_{\vx \in \mx} \arg\max_{\vy \in \my} \objfct(\vx, \vy)$ \Comment{Best minimizer}
        \If {$\objfct(\vx'_*, \vy'_*) < \objfct(\vx_{\lambda, t-1}, \vy_{\lambda, t - 1})$} \Comment{Replace worst minimizer}
        \State $\vx_{\lambda, t-1} \gets \vx'_*$ \Comment{Update population}
        \EndIf
        \State $\vx'_{0}, \vy'_{0} \gets \arg\min_{\vx \in \mx} \arg\max_{\vy \in \my} \objfct(\vx, \vy)$ \Comment{Best maximizer}
        \If {$\objfct(\vx'_{0}, \vy'_{0}) > \objfct(\vx_{\lambda, t-1}, \vy_{\lambda, t-1})$} \Comment{Replace worst maximizer}
        \State $\vy_{\lambda, t-1}\gets \vy'_{0}$ \Comment{Update population}
        \EndIf
                \State $\mx_t \gets \mx_{t-1}$ \Comment{Replace population}
                \State $\my_t \gets \my_{t-1}$ \Comment{Replace population}
        \Until $t \geq T$
        \State $\vx_*, \vy_* \gets \arg\min_{\vx \in \mx_T} \arg\max_{\vy \in \my_T} \objfct(\vx, \vy)$ \Comment{Best minimizer}
        \State \Return $\vx_*, \vy_*$
    \end{algorithmic}
 \end{algorithm}

 \paragraph{Minimax Differential Evolution (\mmde)} The \mmde algorithm introduced by~\citet{qiu2017new} attempts to overcome the limitations of existing approaches in solving minimax optimization problems using differential evolution. They introduce a \emph{bottom-boosting scheme} which skips redundant objective function computations while maintaining the reliability of solution evaluations. Their approach finds solutions with the best worst-case performance rather than computing worst-case scenarios for all the candidates. They realize this insight by modeling their population space as a min-heap. Although they motivate the problem well, their solution is not driven by any theoretical insight.

\section{Methods}
\label{sec:methods}
In this section, we present our proposed framework to use \ES for black-box minimax problems after providing a formal motivation for the same.

\ES are heuristic search methods inspired by natural evolution.  Given a fitness (objective) function, say $f:\Xspace \to \mathbb{R}$, these methods mutate (perturb) a population of genotypes (search points in $\Xspace$) over multiple generations (iterations). At each generation, the fitness of each genotype is evaluated based on $f$. The fittest genotypes among the current population are then recombined to generate the next population. At the end, the genotype (corresponds to a point in $\Xspace$) with the \emph{best} fitness value is returned as the best point that optimizes $f$. The notion of "best" refers to the minimum or maximum obtained value of $f$ in a minimization or maximization setup, respectively. Here, we briefly describe one form of \ES, in particular a simplified version of natural \ES that has recently gained significant attention by the machine learning community~\cite{lehman2017more}. As outlined in Algorithm~\ref{alg:es}, it represents the population with an isotropic Gaussian distribution over the search space $\Xspace$ with mean $\vmu$ and \emph{fixed} covariance $\sigma^2I$ where $I$ is the identity matrix. Over generations, the algorithm aims to maximize the expected fitness value $\mathbb{E}_{\vx\sim \mathcal{N}(\vmu, \sigma^2I)}[f(\vx)]$ with respect to the distribution's mean $\vmu$ via stochastic gradient ascent using a population size of $\lambda$ as shown in Line~\ref{line:perturbations} of Algorithm~\ref{alg:es}, which makes use of the re-parameterization and log-likelihood tricks with $\vx=\vmu + \sigma \vepsilon$~\cite{wierstra2014natural,salimans2017evolution}.: $$\mathbb{E}_{\vx\sim \mathcal{N}(\vmu, \sigma^2I)}[f(\vx)]=\mathbb{E}_{\vepsilon\sim \mathcal{N}(0, I)}[f(\vmu + \sigma \vepsilon)]$$
\begin{equation}
\nabla_{\vmu}\mathbb{E}_{\vx\sim \mathcal{N}(\vmu, \sigma^2I)}[f(\vx)]=
\frac{1}{\sigma}\mathbb{E}_{\vepsilon\sim \mathcal{N}(0, I)}[f(\vmu+ \sigma \vepsilon) \vepsilon]\label{eq:nes-gradient}
\end{equation}

\begin{algorithm}
	\small
	\caption{A Simplified  Example of Evolution Strategy (\ES)\newline
		\textbf{Input:}\newline 
		$~~$\hspace{\algorithmicindent}$\eta$ : learning rate \newline $\sigma$ : perturbation standard deviation, \newline
		$\lambda$ : number of perturbations (population size) \newline
		$~~$\hspace{\algorithmicindent}$T$ : number of iterations (generations), \newline
		$f:\mathcal{X}\to \mathbb{R}$ : fitness function}\label{alg:es}
	\begin{algorithmic}[1]
		\State ${\vmu}_0 \sim \mathcal{U}(\Xspace)$ 
		\For {$t=0$ \textbf{to} $T$}
		\For{$i=1$ \textbf{to} $\lambda$}
		\State $\vepsilon_i\sim \mathcal{N}(0, I)$
		\State $f_i \gets f(\vmu_t + \sigma\vepsilon_i)$
		\EndFor
		\State $\vmu_{t+1} \gets \vmu_{t} + \eta \frac{1}{\lambda \sigma} \sum_{i=1}^{\lambda}f_i \vepsilon_i$\label{line:perturbations}
		\EndFor
	\end{algorithmic}
\end{algorithm}
\paragraph{Descent Direction for Minimax} Next, we show that the direction computed by the random perturbations of the current mean (Line~\ref{line:perturbations} of Algorithm~\ref{alg:es}) can be used to approximate a descent direction of $\max_{\vy} \objfct(\vx,\vy)$. Prior to that and for completeness, we reproduce \cite{madry2017towards}'s proposition A.2 on the application of Danskin's theorem~\cite{danskin1966theory} for minimax problems that are continuously differentiable in $\vx$.
\begin{theorem}[\citet{madry2017towards}]
	Let $\vy^*$ be such that $\vy^* \in \Yspace$ and is a maximizer for$\;\max_{\vy} \objfct(\vx, \vy)$. Then, as long as it is nonzero, $-\nabla_{\vx}\objfct(\vx, \vy^*)$ is a descent direction for~$\;\max_{\vy} \objfct(\vx, \vy)$.
	\label{thm:madry}
\end{theorem}
\begin{proof}
	See \cite{madry2017towards}.
\end{proof}

From Theorem~\ref{thm:madry} and the assumption that $L$ is twice continuously differentiable in $\vx$, we have the following corollary.
\begin{corollary}
	Let $\vy^*$ be such that $\vy^* \in \Yspace$ and is a maximizer for$\;\max_{\vy} \objfct( \vx, \vy)$. Then, for an arbitrary small $\sigma >0$ and $\mathbf{\epsilon}_1,\ldots, \mathbf{\epsilon}_\lambda \sim \mathcal{N}(0,I)$, 
	\begin{equation} -\frac{1}{\sigma \lambda}\sum^\lambda_{i=1}\objfct(\vx+\sigma\mathbf{\epsilon}_i, \vy^*)\mathbf{\epsilon}_i
	\label{eq:monte-carlo-descent}
	\end{equation} is a Monte Carlo approximation of a  descent direction for$\;\max_{\vy} \objfct(\vx, \vy)$.
	\label{corl:descent-direction}
\end{corollary}

\begin{proof}
	Without loss of generality, let $\Xspace \subseteq \mathbb{R}$. Then, since $\sigma$ is arbitrary small, we can approximate $\objfct$ with a second-order Taylor polynomial,
	\begin{equation}\objfct(x+ {\sigma\epsilon}, \vy^*)\approx 
	\objfct(x, \vy^*) + \objfct^\prime(x, \vy^*) \sigma\epsilon  + \frac{1}{2} \objfct^{\prime\prime}(x, \vy^*)\sigma^2\epsilon^2\;.
	\label{eq:tayloer-poly}
	\end{equation}
Based on~\eqref{eq:tayloer-poly} and the linearity of expectation, the expectation of $\objfct(x+\sigma\epsilon, \vy^*)\epsilon$ with respect to $\epsilon\sim \mathcal{N}(0, 1)$, $\mathbb{E}_{\epsilon}[\objfct(x+\sigma\epsilon, \vy^*)\epsilon]$,  can be written as
\begin{eqnarray}
\underbrace{\mathbb{E}_{\epsilon}[\objfct(x, \vy^*) \epsilon]}_{0}+ \underbrace{\mathbb{E}_{\epsilon}[ \objfct^\prime(x, \vy^*)\sigma\epsilon^2]}_{ \objfct^\prime(x, \vy^*)\sigma} + \frac{1}{2}
\underbrace{\mathbb{E}_{\epsilon}[  \objfct^{\prime\prime}(x, \vy^*) \sigma^2\epsilon^3]}_{0}\;,
\end{eqnarray}
where the values of the terms (written under the corresponding braces) come from the definition of \emph{central moments} of the Gaussian distribution~\cite{winkelbauer2012moments}. That is, $\mathbb{E}_{\epsilon}[\objfct(x+\epsilon, \vy^*)\epsilon]\approx\objfct^\prime(x, \vy^*)\sigma$. Thus, $-\objfct^\prime(x, \vy^*)$, which is---from Theorem~\ref{thm:madry}---a descent direction for  $\max_{\vy} \objfct(x, \vy)$, has a Monte Carlo estimation of the form
	\begin{eqnarray}
	-\objfct^\prime(x, \vy^*) &=& -\frac{1}{\sigma}\mathbb{E}_{\epsilon}[\objfct(x+\sigma\epsilon, \vy^*)\epsilon] \nonumber\\
	& \approx & -\frac{1}{\sigma \lambda}\sum^\lambda_{i=1}\objfct(x+\sigma\mathbf{\epsilon}_i, \vy^*)\mathbf{\epsilon}_i\;, \;\;\epsilon_i \sim \mathcal{N}(0, 1)\;, \nonumber
	\end{eqnarray}
as stated in~\eqref{eq:monte-carlo-descent}.\end{proof}
\begin{remark}
	Although Theorem~\ref{thm:madry} assumes  $\objfct$ to be continuously differentiable in $\vx$, it has been shown empirically in~\cite{madry2017towards} that breaking this assumption is not an issue in practice. 
\end{remark}

\begin{remark}
	Current state-of-the-art \ES algorithms are far more than stochastic gradient estimators due to their  \begin{inparaenum}[i)]
		\item ability to automatically adjust the scale on which they sample (step size adaptation)
		\item ability to model second order information (covariance matrix adaptation)
		\item invariance to rank-preserving transformations of objective values.
	\end{inparaenum} That is, they do not estimate the gradient, but a related quantity~\cite{JMLR:v18:14-467}. Our introduction of the simplified version~(Algorithm~\ref{alg:es}) was to show that a simplified \ES algorithm can conform to the guarantees of Theorem~\ref{thm:madry}. In the rest of the paper, we consider established \ES variants.
\end{remark}

\paragraph{Approximating Inner Maximizers} While Corollary~\ref{corl:descent-direction} motivated the use of \ES to approximate descent directions for $\max_{\vy \in \Yspace}\objfct(\vx,\vy)$, an inner maximizer must be computed beforehand. \ES can be used for the same. In other words, our use of \ES will be of two-fold: 1) computing an inner maximizer $\vy^*$ for $\objfct$; followed by 2) approximating a descent direction on $\vx$ for the outer minimization problem along which we proceed to compute the next inner maximizer, and the decent direction therein. 

However, the inner maximization problem of~\eqref{eq:minmax} can be non-concave, for which \ES may converge to a local inner maximizer.  Restart techniques are common in gradient-free optimization to deal with such cases~\cite{loshchilov2012alternative,hansen2009benchmarking}. Therefore, we use \ES with restarts when computing inner maximizers.

\paragraph{Convergence} Up to this point, we have seen how \ES can be  used iteratively to approximate descent directions at inner maximizers of~\eqref{eq:minmax}. Furthermore, we proposed to address the issue of non-concavity of the inner maximization problem through restarts. One fundamental question is \emph{how much do we step along the descent direction given an inner maximizer?} If we step too much or little, Corollary~\ref{corl:descent-direction} might not be of help anymore and we may get stuck in cycles similar to the non-convergent behavior of cyclic coordinate descent on certain problems~~\cite{powell1973search}. We investigate this question empirically in our experiments. 

One should note that cyclic non-convergent behavior is common among coevolutionary algorithms on \emph{asymmetrical} minimax problems~\cite{qiu2017new}. Furthermore, the outer minimization problem can be non-convex. Since we are using \ES to approximate the gradient (and eventually the descent direction), we resort to gradient-based restart techniques~\cite{loshchilov2016sgdr} to deal with cycles and non-convexity of the outer minimization problem. In particular, we build on Powell's technique~\cite{powell1977restart} to restart whenever the angle between momentum of the outer minimization and the current descent direction is non-positive.

One can observe that we employ gradient-\emph{free} restart techniques to solve the inner maximization problem and gradient-\emph{based} counterparts for the outer minimization problem. This is in line with our setup where computing an outer minimizer is guided by the gradient as a descent direction, while approximating the gradient is not a concern for computing an inner maximizer.

\paragraph{ES Variants} While the aforementioned discussion has been with respect to a very simplified form of \ES (Algorithm~\ref{alg:es}), there are other variants in the \ES family that differ in how the population is represented, mutated, and recombined. For instance, \emph{antithetic} (or \emph{mirrored}) sampling~\cite{brockhoff2010mirrored} can be incorporated to evaluate pairs of perturbations~$\vepsilon, -\vepsilon$ for every sampled $\vepsilon\sim \mathcal{N}(0,I)$. This was shown to reduce variance of the stochastic estimation (Line ~\ref{line:perturbations} of Algorithm~\ref{alg:es}) of ~\eqref{eq:nes-gradient}~\cite{salimans2017evolution}. Moreover, fitness shaping~\cite{wierstra2014natural} can be used to ensure invariance with respect to order-preserving fitness transformations. 

Instead of a fixed covariance $\sigma^2I$, the Covariance Matrix Adaptation Evolution Strategy (\cmaes) represents the population by a full covariance multivariate Gaussian~\cite{hansen2001completely}. Several theoretical studies addressed the analogy of \cmaes and the natural gradient ascent on the parameter space of the Gaussian distribution~\cite{akimoto2012analysis, glasmachers2010exponential}. 

In our experiments, we employ some of the aforementioned \ES variants in the \horse framework and compare their efficacy in computing a descent direction for the outer minimization problem

\begin{algorithm}
	\small
	\caption{\horse \newline
		\textbf{Input:}\newline 
		$~~$\hspace{\algorithmicindent}$T$: number of iterations,\newline~$~~$\hspace{\algorithmicindent}$v$: number of function evaluations (FEs) per iteration\newline
		~$~~$$s\in (0,0.5]$: budget allocation for descent direction}
	\label{alg:reckless}
	\begin{algorithmic}[1]
		\State $\vx_0 \sim \mathcal{U}(\Xspace)$
		\State $\vy_0 \sim \mathcal{U}(\Yspace)$ 
		\State $\vx_* \gets \vx_0$
		\State $\vy_* \gets \vy_0$
		\For{$t=1$~\textbf{to}~$T$}		
		\State $\vy_t \gets \arg\max_{\vy \in \Yspace} \objfct(\vx_{t-1}, \vy)$ by \ES with restarts~and~$(1-s)v$~FEs.\label{line:max}
		\If {$\objfct(\vx_{t-1}, \vy_{t}) < \objfct(\vx_*, \vy_*)$}\label{line:best-if}\Comment{best solution}
		\State $\vx_*\gets \vx_{t-1}$
		\State $\vy_*\gets \vy_t$
		\EndIf\label{line:best-end-if}
		\State $\vx_{t} \gets \arg\min_{\vx \in \Xspace} \objfct(\vx, \vy_t)$ by \ES with $sv$ FEs. \label{line:min}
		\If{ $(x_t-x_{t-1})^T(x_{t-1}-x_{t-2}) \leq 0$}\label{line:restart} \Comment{restart condition}
		\State $\vx_t \sim \mathcal{U}(\Xspace)$
		\State $\vy_t \sim \mathcal{U}(\Yspace)$ 
		\EndIf\label{line:restart-end}
		\EndFor
		\State \Return $\vx_*, \vy_*$
	\end{algorithmic}
\end{algorithm}

\paragraph{\horse for Black-Box Minimax Problems} Based on the above, we can  now present~\horse, our optimization framework for black-box minimax problems. As shown in Algorithm~\ref{alg:reckless}, the framework comes with 3 parameters: the number of iterations $T$, the number of function evaluations per iteration $v$, and the last parameter $s\in(0,0.5]$ which, along with $v$,  controls how much we should step along the descent direction for the outer minimization problem. Controlling the descent step is expressed in terms of the number of function evaluations per iteration to make the framework independent of the \ES algorithm used: some \ES variants (e.g., Algorithm~\ref{alg:es}) use a fixed learning rate and perturbation variance whereas others update them in an adaptive fashion. \horse starts by randomly sampling a pair $(\vx_0, \vy_0) \in \Xspace \times \Yspace$. It then proceeds to look for an inner maximizer using \ES with restarts using $(1-s)v$ function evaluations (Line~\ref{line:max}). While this is a prerequisite to step along a descent direction using $sv$ function evaluations (Line~\ref{line:min}), it also can be used to keep record of the best obtained solution so far $\vx_*$ and its "worst" maximizer $\vy_*$ as shown in Lines~\ref{line:best-if}--\ref{line:best-end-if}. Line~\ref{line:restart} check for the gradient-based restart condition for the outer minimization comparing the direction of the current descent $(\vx_t -\vx_{t-1})$ against that of the momentum $(\vx_{t-1}-\vx_{t-2})$. Note that the restart condition is checked only  when $\vx_{t-2},\vy_{t-2}$ (and subsequently $\vx_{t-1},\vy_{t-1}$) are valid. That is, for iterations right after a restart or iteration $t=1$, the restart check (Lines~\ref{line:restart}--\ref{line:restart-end}) does not take place.
\section{Experiments}
\label{sec:experiments}

To complement its theoretical perspective, this section presents a numerical assessment of \horse. First, we investigate the questions raised in Section~\ref{sec:methods} about \emph{\ES variants for approximating a descent direction for the outer minimization} and \emph{steps along the descent direction} on a set of benchmark problems. Second, we compare \horse with established coevolutionary algorithms for black-box minimax problems on the same set of problems in terms of scalability and convergence, given different function evaluation budgets. Finally, the applicability of \horse is demonstrated on a real-world example of digital filter design.
 
\subsection{Setup}
The six problems used in~\cite{qiu2017new} as defined in Table~\ref{tbl:benchmark-problems} are selected for our benchmark. In terms of symmetry, problems $\objfct_1$, $\objfct_5$ and $\objfct_6$ are symmetrical, while the rest are asymmetrical. On the other hand, only $\objfct_1$ and $\objfct_2$ are scalable with regard to dimensionality: $n_x$ and $n_y$. All the experiments were carried out on an Intel Core i7-6700K CPU @ $4.00GHz\times 8$ 64-bit Ubuntu machine with $64GB$ RAM. For statistical significance, each algorithm was evaluated over $60$ independent runs for each problem instance. The regret~\eqref{eq:regret} is used as the performance metric in a \emph{fixed-budget approach}~\cite{hansen2016coco}.  The first term of~\eqref{eq:regret} was set to the maximum of the  values returned by \texttt{basinhopping}, \texttt{differential\_evolution} from the \textsf{SciPy} Python library as well as \cmaes~\cite{pycmaes}, each set to their default parameters. Let \fes denote the number of function evaluations given to solve the black-box minimax problem~\eqref{eq:minmax}. In line with~\cite{qiu2017new}, we set its maximum to $10^5$, except for the scalability experiments where the number of function evaluations given grows with the problem dimensionality as $10^4\times n_x$ with $n_x=n_y$.

\paragraph{\horse Variants} We consider two main \ES variants for outer minimization (Line~\ref{line:min} of Algorithm~\ref{alg:reckless}): The first is Algorithm~\ref{alg:es}, which we refer to as \nes in the rest of this work. We used standardized fitness values---that is, we subtract the mean and divide by the standard deviation of the obtained function values for all the $\lambda$ perturbations. 
We set the learning rate to $\eta=10^{-4}$ and $\sigma=0.25$ ($1/4$th of the search domain width~\cite{pycmaes}). The second variant is the state of the art \cmaes algorithm (without restart). We investigate the effectiveness of antithetic sampling as well as gradient-based restart (Lines~\ref{line:restart}--\ref{line:restart-end} of Algorithm~\ref{alg:reckless}) on both these algorithms. As a result, we have \emph{eight} variants of \horse, which we notate using the following symbols: C for \cmaes; N for \nes; R for restart; and A for antithetic sampling. For instance, \texttt{ACR} indicates that we use antithetic \cmaes in Line~\ref{line:min} of Algorithm~\ref{alg:reckless} and that the gradient-based restart (Lines~\ref{line:restart}--\ref{line:restart-end} of Algorithm~\ref{alg:reckless}) is enabled.  Similarly, \texttt{AC} denotes the same but with gradient-based restart being disabled. For all our experiments, 
\horse's inner maximizer is set to \cmaes with restarts~\cite{hansen2009benchmarking} (Line~\ref{line:max} of Algorithm~\ref{alg:reckless}).

\paragraph{Steps along the Descent Direction} As mentioned in Section~\ref{sec:methods}, we control how much we descend given an inner maximizer by varying $sv$, the number of function evaluations allocated for outer minimization per iteration---Line~\ref{line:min} of Algorithm~\ref{alg:reckless}. In our experiments, five values were tested, namely $s\in S=\{0.1, 0.2, 0.3, 0.4, 0.5\}$. Given the total number of function evaluations \fes given to \horse, the number of iterations $T$ and the evaluation budget per iteration $v$ are computed as follows.
\begin{equation}
T=\Bigg\lceil\sqrt{\frac{\fes}{(|S|+1)( \lambda_1 + 2s \lambda_2)}}\Bigg\rceil\;; v=\bigg\lfloor\frac{\fes}{T}\bigg\rfloor\;,
\label{eq:t}
\end{equation}
where $\lambda_1$ is the population size of \ES for inner maximization (Line~\ref{line:max}), and $\lambda_2$ is the population size of \ES for outer minimization (Line~\ref{line:min}). We borrowed their settings from \cmaes~\cite{pycmaes}: $\lambda_1 = 4 + \lfloor3 \log n_y \rfloor$ and $\lambda_2 = 4 + \lfloor3 \log n_x \rfloor$. The number of iterations $T$~\eqref{eq:t} can be viewed as the square root of rounds, in which both, the inner maximizer and the outer minimizer \ES evolve $|S|$ times (+1 for initializing the population) given the number of function evaluations \fes. This setup yields a noticeable difference in the number of function evaluations $sv$ allocated for the outer minimization over the $5$ considered values $S$. Table~\ref{tbl:s-values} provides an example of the parameters setup given $s$ and \fes.

\paragraph{Comparison Setup}
We compare \horse against three established algorithms for black-box minimax optimization, namely, Coevolution Alternating (\coeva), Coevolution Parallel (\coevp), and Minimax Differential Evolution (\mmde). The \textit{mutation rate}, \textit{population size}, and \textit{number of candidates replaced per generation} were set to $0.9, 10, 2$  - values which we had tuned for similar problems we had worked with in the past. Carefully tuning them for this specific test-bench is a direction for future work. For \mmde, we implemented the algorithm described in \cite{qiu2017new} in Python. We use the recommended hyperparameter values presented in their work. Specifically, $K_s = 190$, \textit{mutation probability}$=0.7$, \textit{crossover probability}$=0.5$, and \textit{population size} $= 100$.

\subsection{Results}
We present the results of our four experiments in the form of \emph{regret convergence}\footnote{Note that regret is the measure defined in Eq.~\eqref{eq:regret}. This is \emph{not} the MSE of Eq.~\eqref{eq:mse}.}  and \emph{critical difference} (CD) plots. In each of the regret plots, the dashed, bold lines represent the mean regret over $60$ independent runs, and are surrounded by error bands signifying one standard deviation from the mean. To confirm that the regret values we observe are not chance occurrences, tests for statistical significance are carried out. 

\paragraph{\horse Steps along the Descent Direction} Figure~\ref{fig:p-regret} shows the performance of \horse over different budget allocations $sv$ for outer minimization, where $s\in S$. For this experiment, we considered only the variant \texttt{N} (\nes without restart) because the number of function evaluations used by \nes are exactly governed by $sv$, whereas \cmaes might overspend or underutilize the given budget (due to early stopping). For high \fes, we observe that $s=0.3$ is the most robust value across the benchmark problems. This suggests that the outer minimization can do well with half of the budget given to the inner maximization. This asymmetric setup for budget allocation is an interesting direction of investigation for coevolutionary algorithms, where the two populations are typically given the same budget of function evaluation. Since we are comparing \horse with coevolutionary algorithms, we set $s=0.5$ for the rest of the experiments.

\paragraph{\horse Variants} The performance of \horse using eight variants of \ES for the outer minimization (Line~\ref{line:min} of Algorithm~\ref{alg:reckless}) is shown in Figure~\ref{fig:reckless-regret}. For low \fes, we observe no difference in the performance. Towards high \fes, variants with restart perform marginally better. No significant difference in performance was observed for antithetic variants over non-antithetic counterparts. It is interesting to note that \cmaes-based variants perform exceptionally well on the symmetrical, quadratic  problem $\objfct_1$. This is expected as \cmaes iteratively estimates a positive definite matrix which, for convex-quadratic functions, is closely related to the inverse Hessian. For the reset of the experiments, we use the variant \texttt{CR}.

\paragraph{Comparison of Algorithms for Minimax} Regret convergence of \horse and the considered coevolutionary algorithms, given different \fes, is shown in Figure~\ref{fig:feval-regret}. We re-iterate that the algorithms are re-run for each of the considered \fes in a fixed-budget approach~\cite{hansen2016coco}, rather than setting the evaluation budget to its maximum and recording the regret at the \fes of interest. With more \fes, the regret of \horse and \mmde consistently decrease. This is not the case for \coeva and \coevp due to their memoryless nature.  In terms of problem dimensionality, \mmde scales poorly in comparison to the other algorithms as shown in Figure~\ref{fig:scale-regret}, whereas \horse's performance is consistent on $\objfct_1$ and comparable with that of \coeva on $\objfct_2$ outperforming both \mmde and \coevp.
 
\paragraph{Statistical Significance} We employ the non-parametric Friedman test, coupled with a post-hoc Nemenyi test with $\alpha=0.05$~\cite{demvsar2006statistical}. The tests' outcome is summarized in the CD plots shown in Figure~\ref{fig:cd}. Overall, we infer that the experimental data is not sufficient to reach any conclusion regarding the differences in the algorithms' performance. A statistically significant difference was only found  between \horse and \coevp in the regret convergence experiment (Figure~\ref{fig:feval-regret}). We hope that with more benchmark problems, statistical significance can be established, and we leave this for future work.

\subsection{Application Example: Digital Filter Design}
 We apply \horse and the other algorithms discussed in this work on an application presented by~\citet{charalambous1979acceleration}. The work introduces the problem of designing a digital filter such that its amplitude $|\mathcal{H}(x,\theta)|$ approximates a function $\mathcal{S(\psi)}$ = $|1-2\psi|$. $|\mathcal{H}(x,\theta)|$ is defined as:
 \begin{equation}
 A \prod\limits_{k=1}^{K}\Big(\frac{1 + a_k^2 + b_k^2 + 2b_k(2cos^2\theta - 1) +  2a_k(1 + b_k)cos\theta}  {1 + c_k^2 + d_k^2 + 2d_k(2cos^2\theta - 1) +  2c_k(1 + d_k)cos\theta}\Big)^\frac{1}{2} \nonumber
 \end{equation}
 The authors then set this up as the following minimax problem:
 \begin{equation}
 \min\limits_{\vx} \max\limits_{0\leq\psi\leq1} |e(x, \psi)|, \text{where}\; e(x, \psi) = |\mathcal{H}(x,\theta)| - \mathcal{S(\psi)}  \nonumber
 \end{equation}
 In this formulation, $\vx=[A, a_1, a_2, \ldots, d_1, d_2]$ ($K = 2$ as chosen by the authors) to minimize and 1 variable ($\psi$) to maximize, i.e., $n_x=9$ and $n_y=1$ with $\Xspace=[-1,1]^{n_x}$, $\Yspace=[0,1]^{n_y}$. We evaluate the four algorithms on this minimax problem. The performance of the algorithms is expressed in terms of regret. Lower the regret, higher the value of $\psi$ explored for a given minimum $x$.  Table \ref{tbl:application} records the regrets of the best points calculated by the four algorithms and the optimal point reported in \cite{charalambous1979acceleration}. We evaluate each algorithm 60 times, for $10^5$ function evaluations, and report here median values. We see that \horse outperforms all other algorithms in finding the best worst-case solution to the digital filter problem.
  
  \begin{table}[h]
    \caption{Regret of the algorithms for the digital filter design.}
    \label{tbl:application}
    \resizebox{0.45\textwidth}{!}{
        \begin{tabular}{c c c c c}
            \toprule
              \textbf{\cite{charalambous1979acceleration}'s method} & \textbf{\textsc{Reckless}} & \textbf{\mmde} & \textbf{\coeva} & \textbf{\coevp} \\
            \midrule
             7.14$\times10^{-3}$ & \textbf{4.16}$\mathbf{\times10^{-13}}$ & 5.09$\times10^{-3}$ & 5.73$\times10^{-6}$ & 4.36$\times10^{-3}$
            \\
            \bottomrule
        \end{tabular}}
 \end{table}
\section{Conclusions}
\label{sec:conclusions}

In this paper, we presented \horse: a theoretically-founded framework tailored to black-box problems that are usually solved in a coevolutionary setup. Our proposition employed the stochastic gradient estimation of \ES motivated by the experimental success of using gradients at approximate inner maximizers of minimax problems. 

As demonstrated on scalable benchmark problems and a real-world application, \horse outperforms the majority of the established coevolutionary algorithms, particularly on high-dimensional problems.  Moreover, we found that minimax problems can be solved where outer minimization is given half of the inner maximization's evaluation budget. Due to the limited number of evaluators (benchmark problems), statistical significance could not be established. In our future work, we hope that a larger set of benchmark problems can address this issue.

\section*{Acknowledgment}

This work was supported by the MIT-IBM Watson AI Lab
and CSAIL CyberSecurity Initiative.

\bibliographystyle{ACM-Reference-Format}
\bibliography{bibliography}

\begin{figure*}[h!]
	\begin{tabular}{ccc}
		\includegraphics[width=0.25\textwidth]{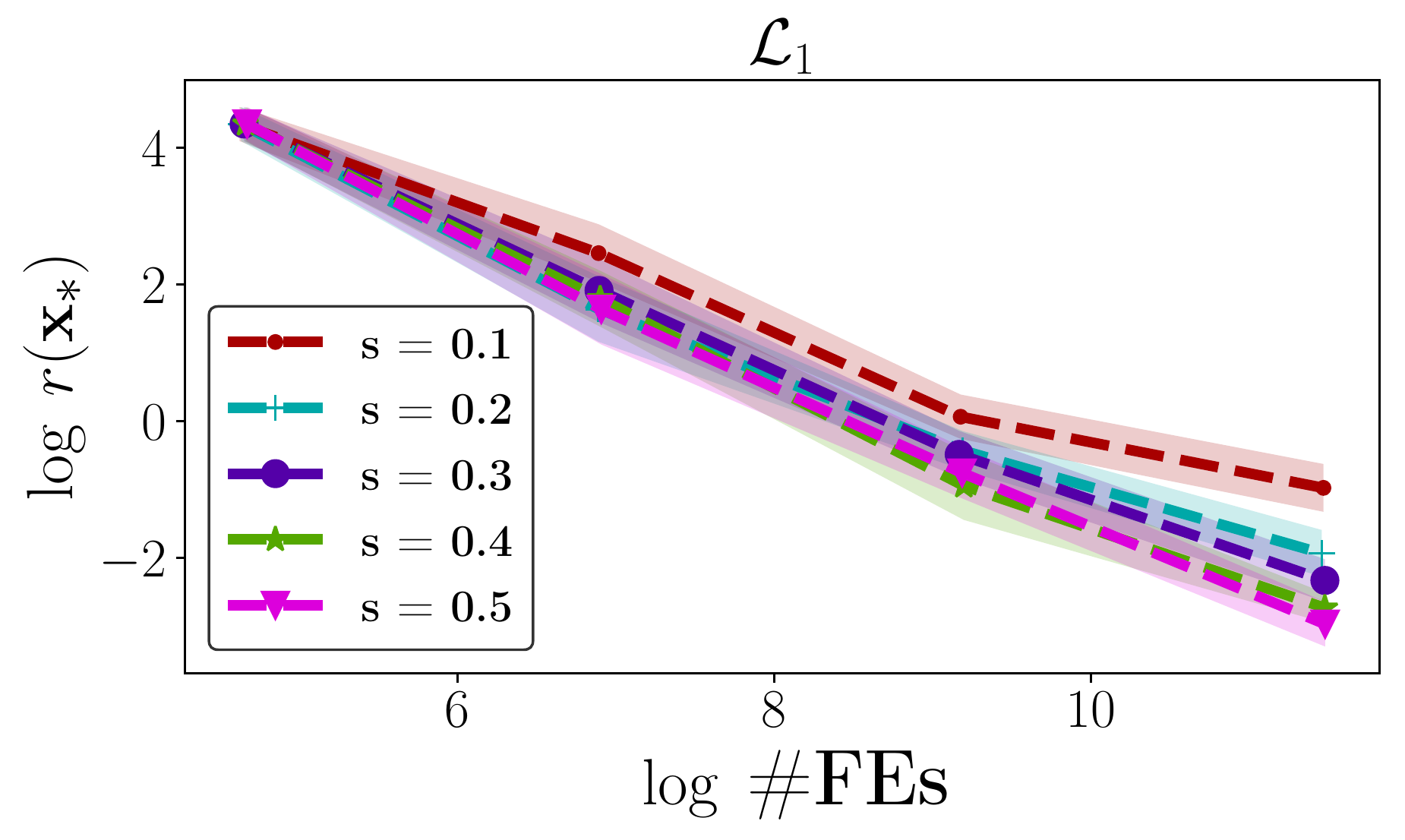} & \includegraphics[width=0.25\textwidth]{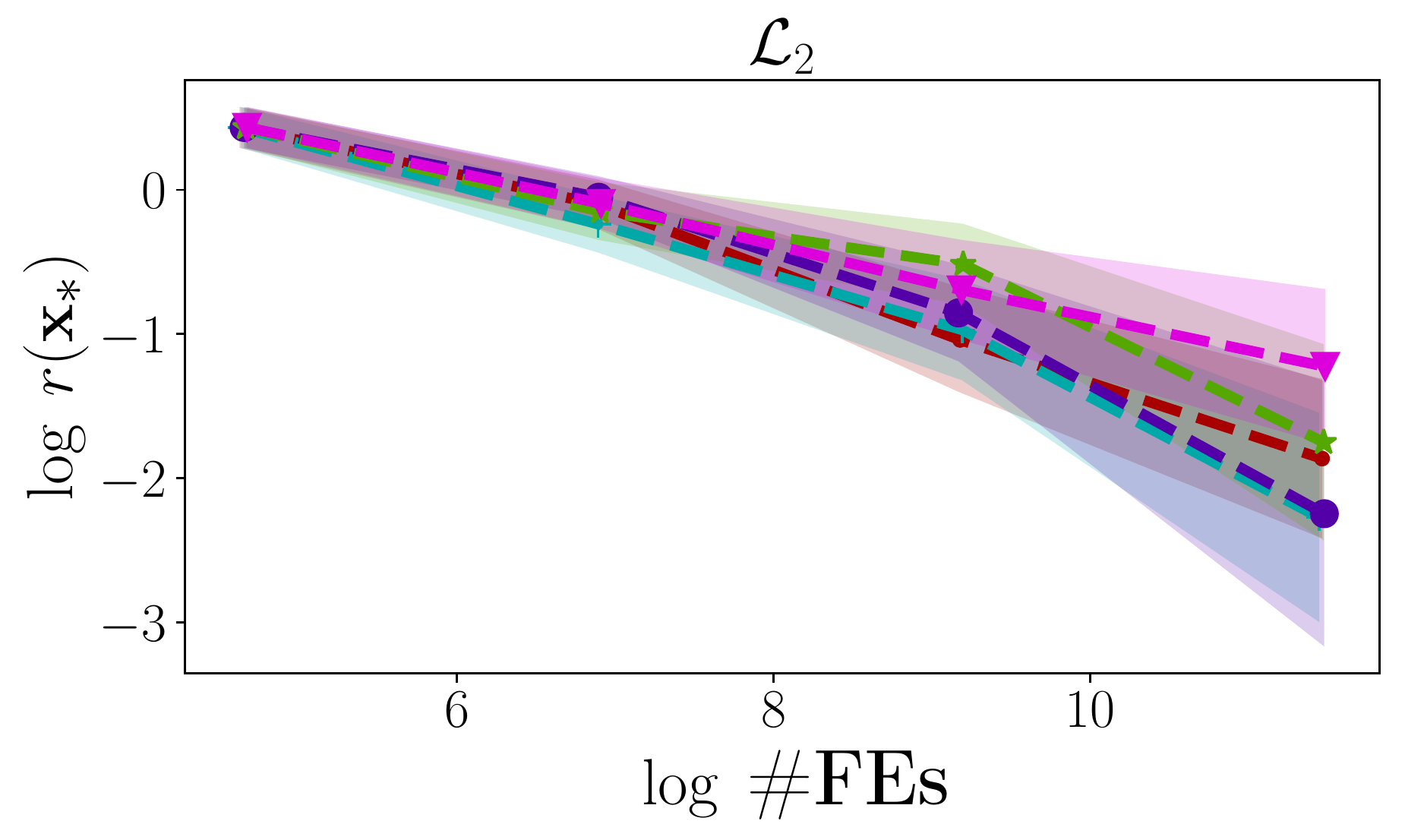} & \includegraphics[width=0.25\textwidth]{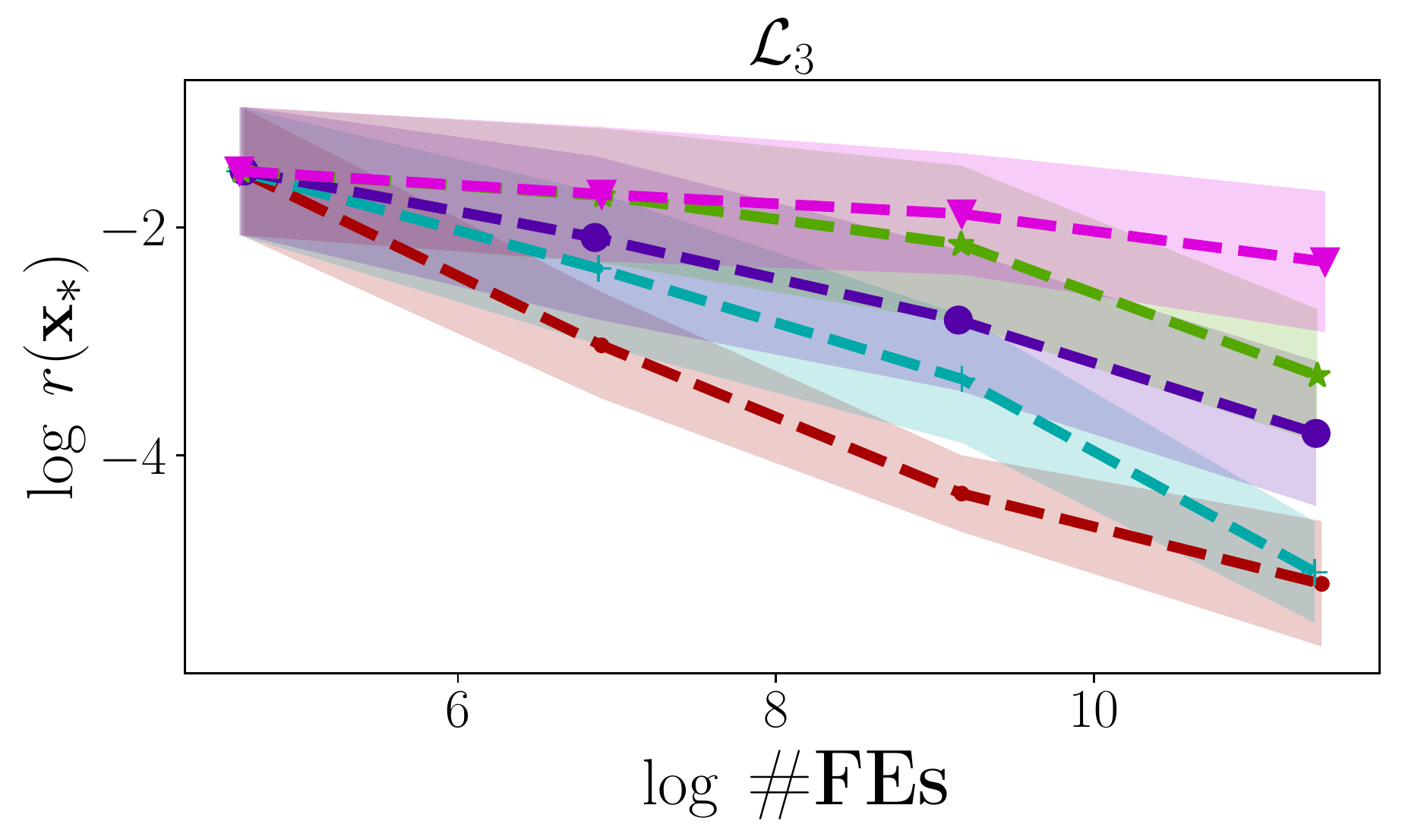}\\
		\includegraphics[width=0.25\textwidth]{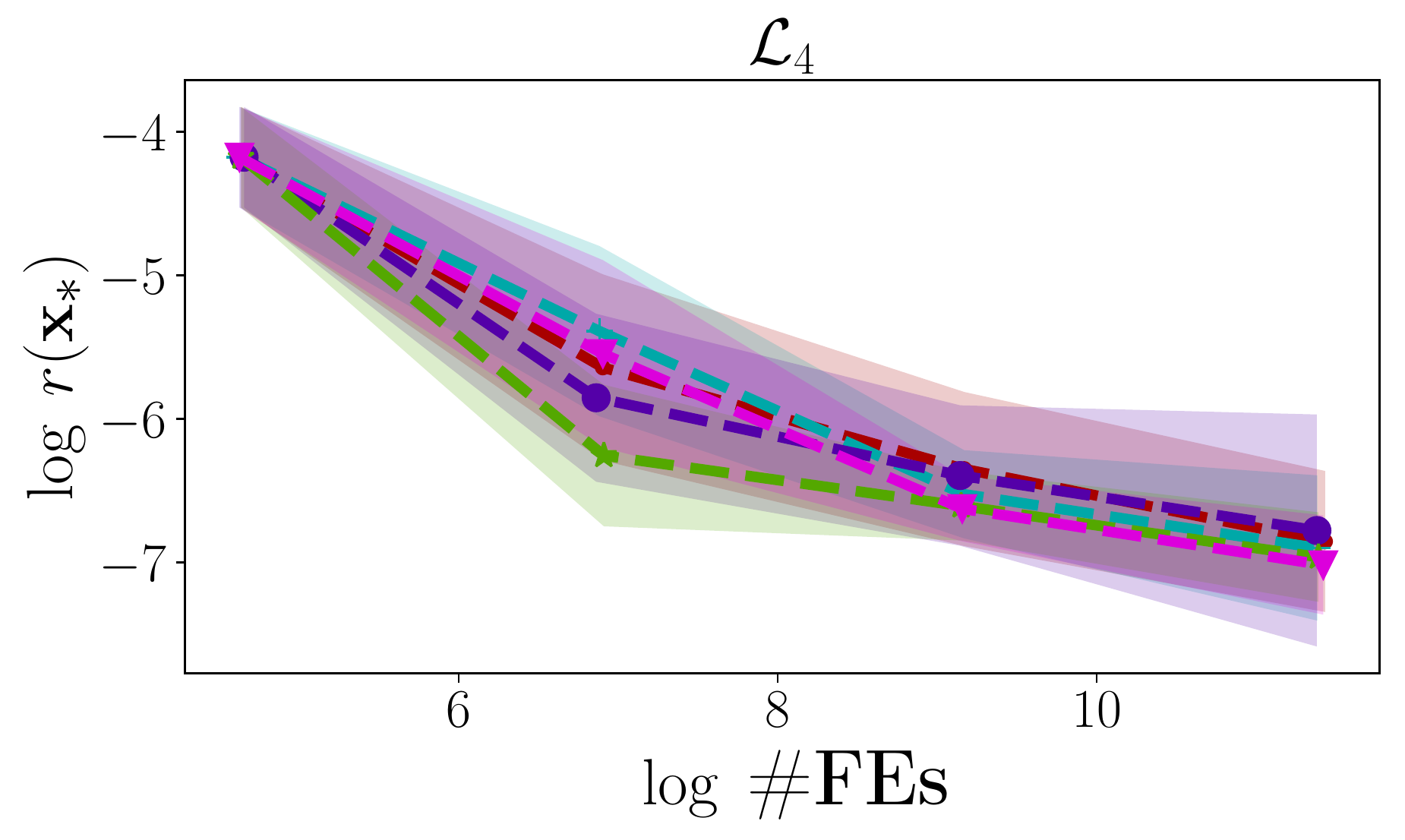} & \includegraphics[width=0.25\textwidth]{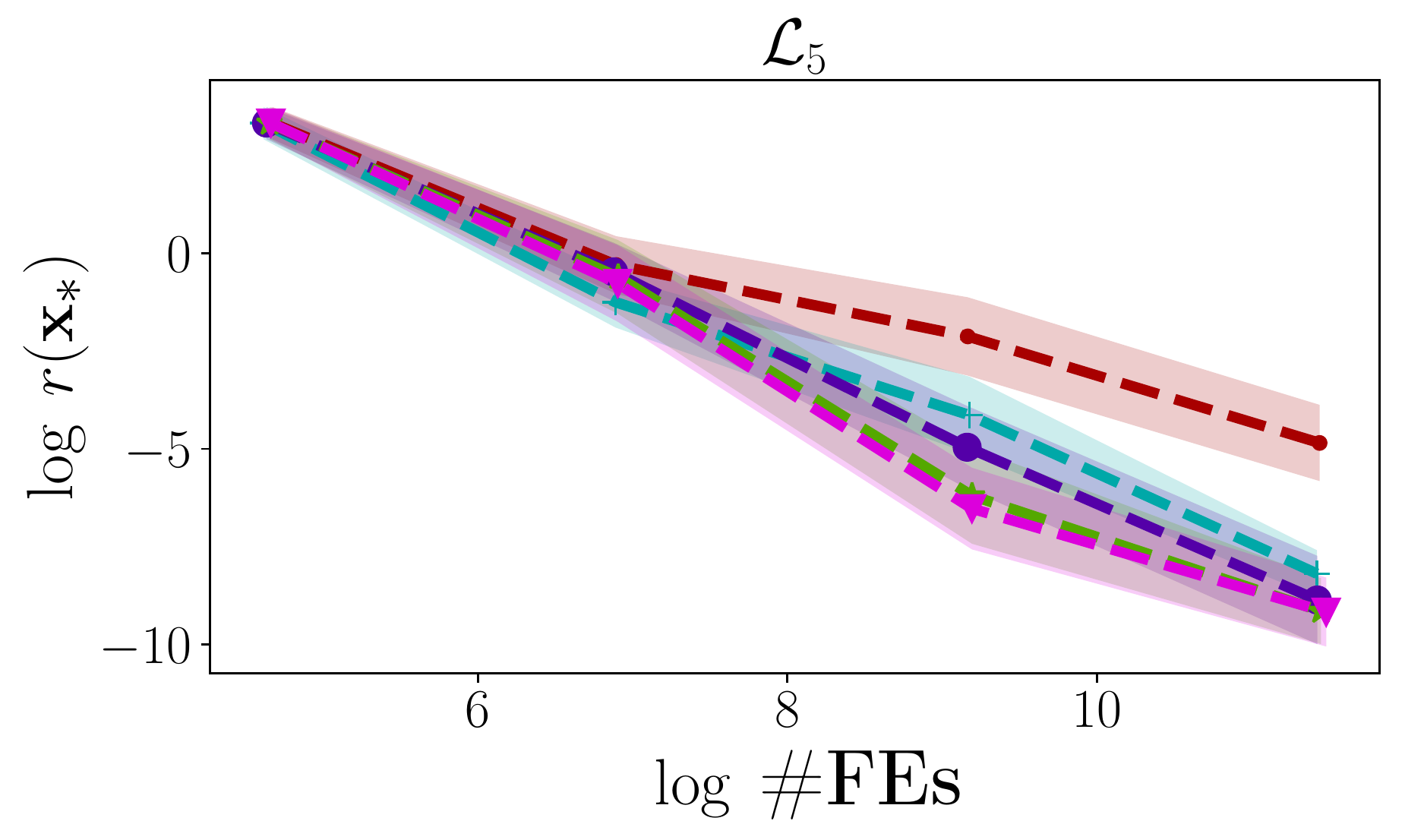} & \includegraphics[width=0.25\textwidth]{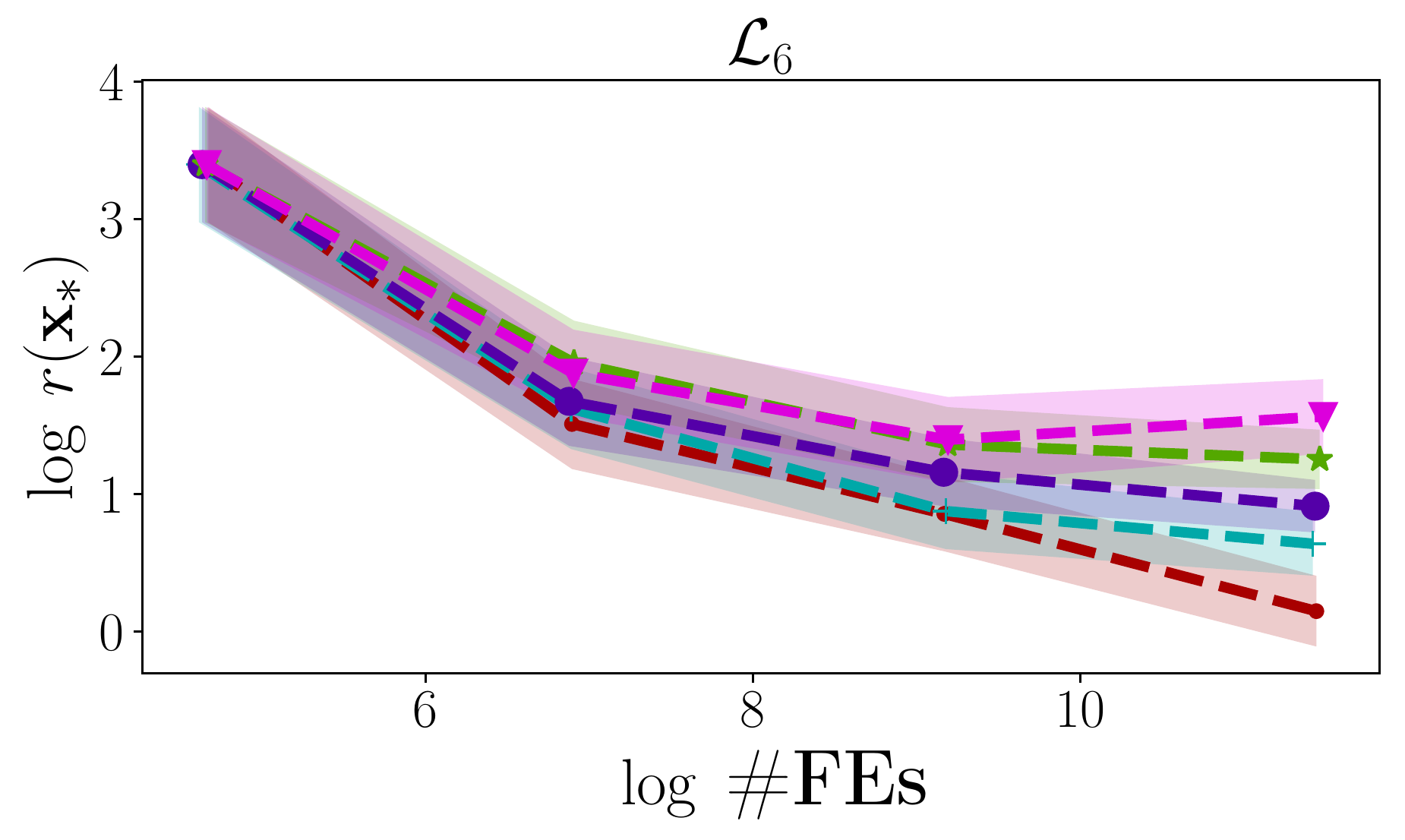}\\
	\end{tabular}
	\caption{\horse Steps along the Descent Direction. The markers indicate the average regret value surrounded by error bands signifying one standard deviation, obtained using $60$ independent runs of the variant \texttt{N}: \nes without restart.}
	\label{fig:p-regret}
\end{figure*}

\begin{figure*}[h!]
	\begin{tabular}{ccc}
		\includegraphics[width=0.25\textwidth]{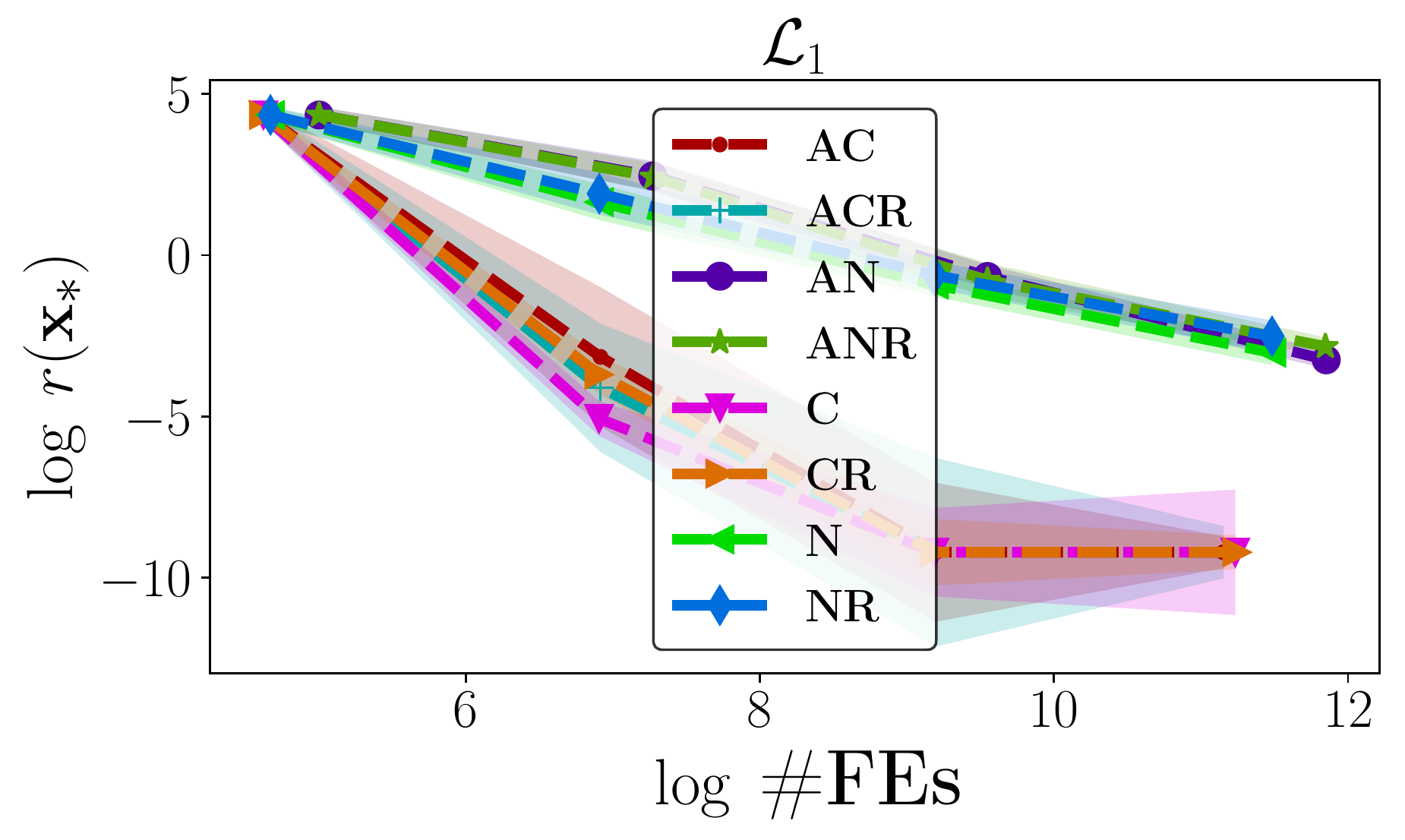} & \includegraphics[width=0.25\textwidth]{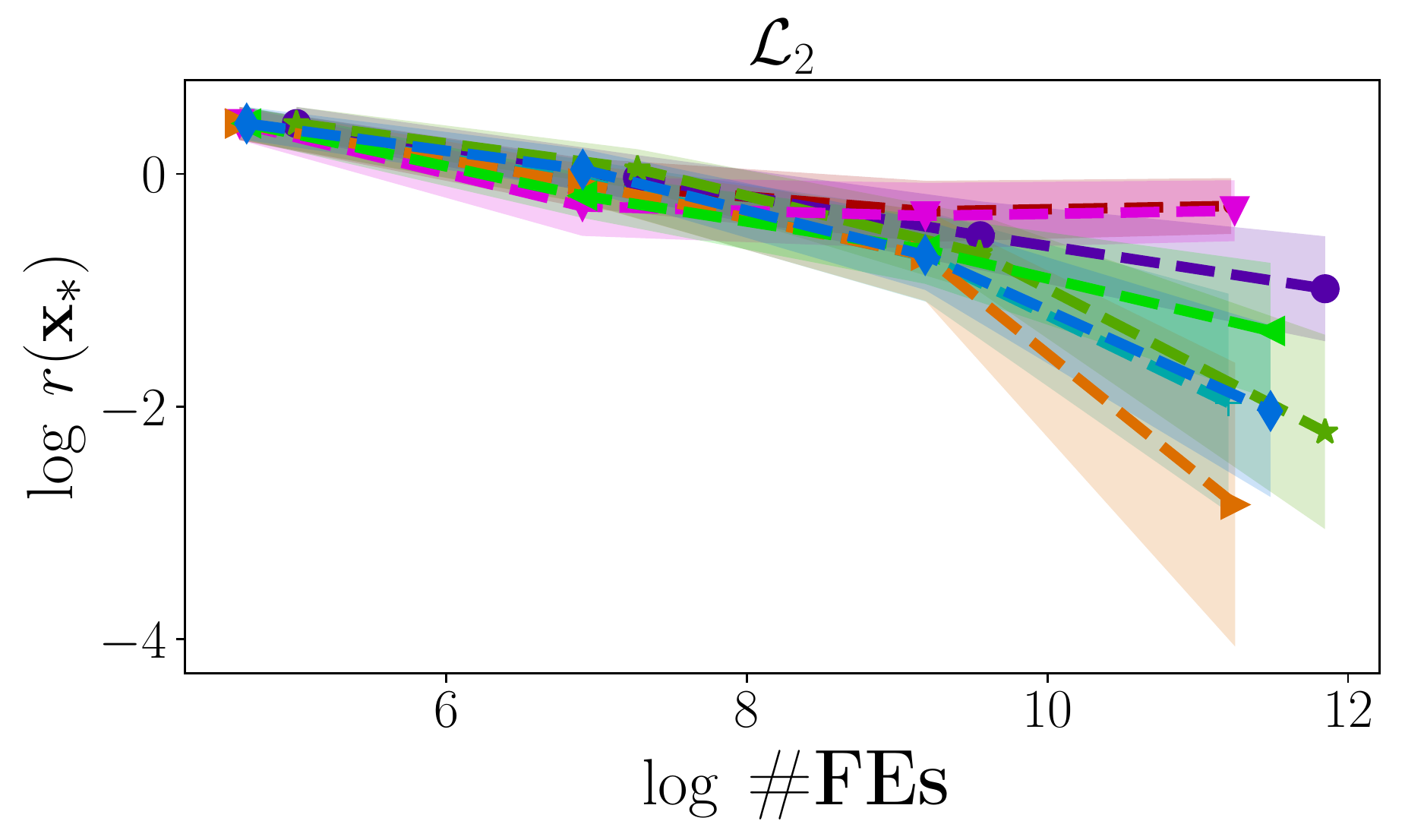} & \includegraphics[width=0.25\textwidth]{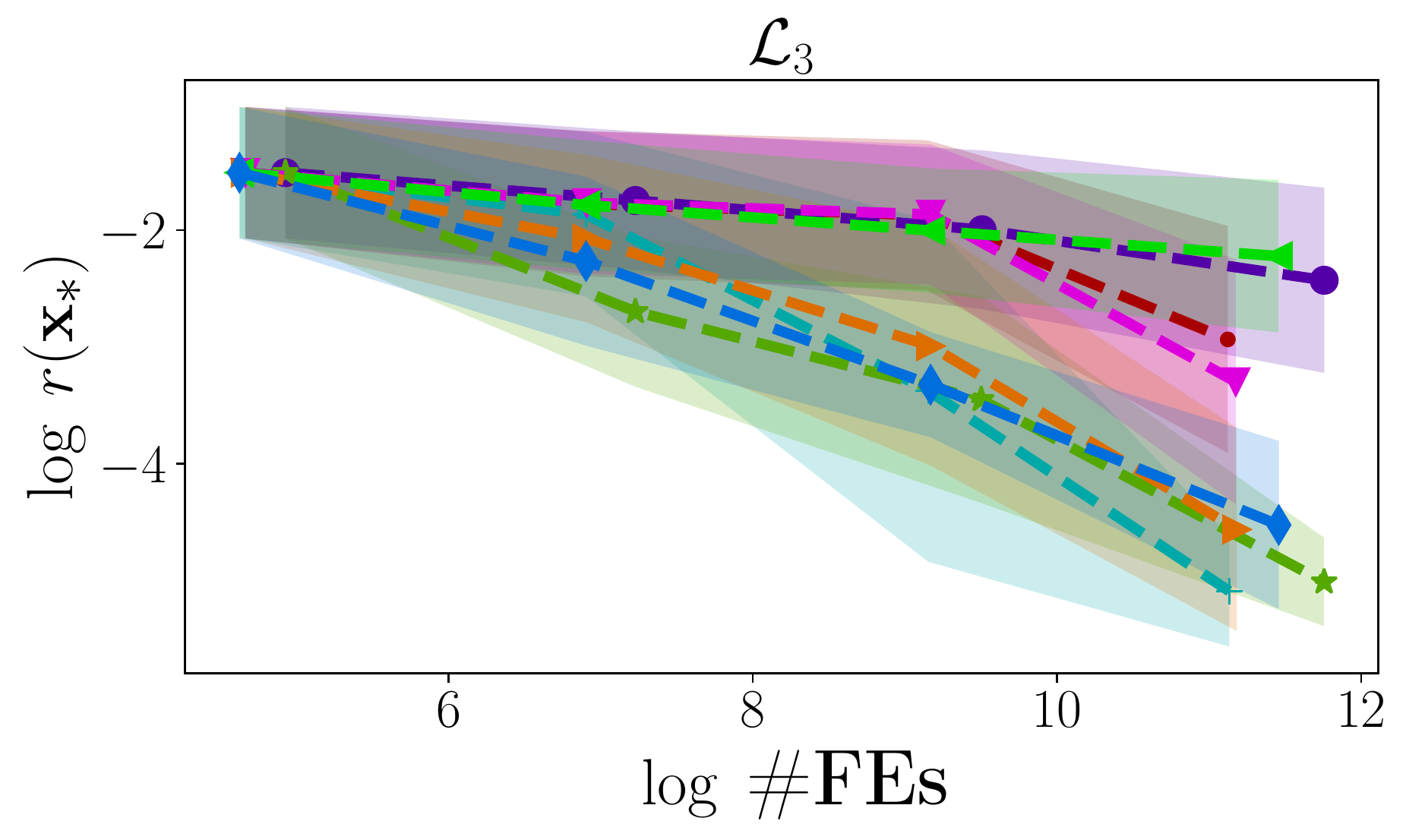}\\
		\includegraphics[width=0.25\textwidth]{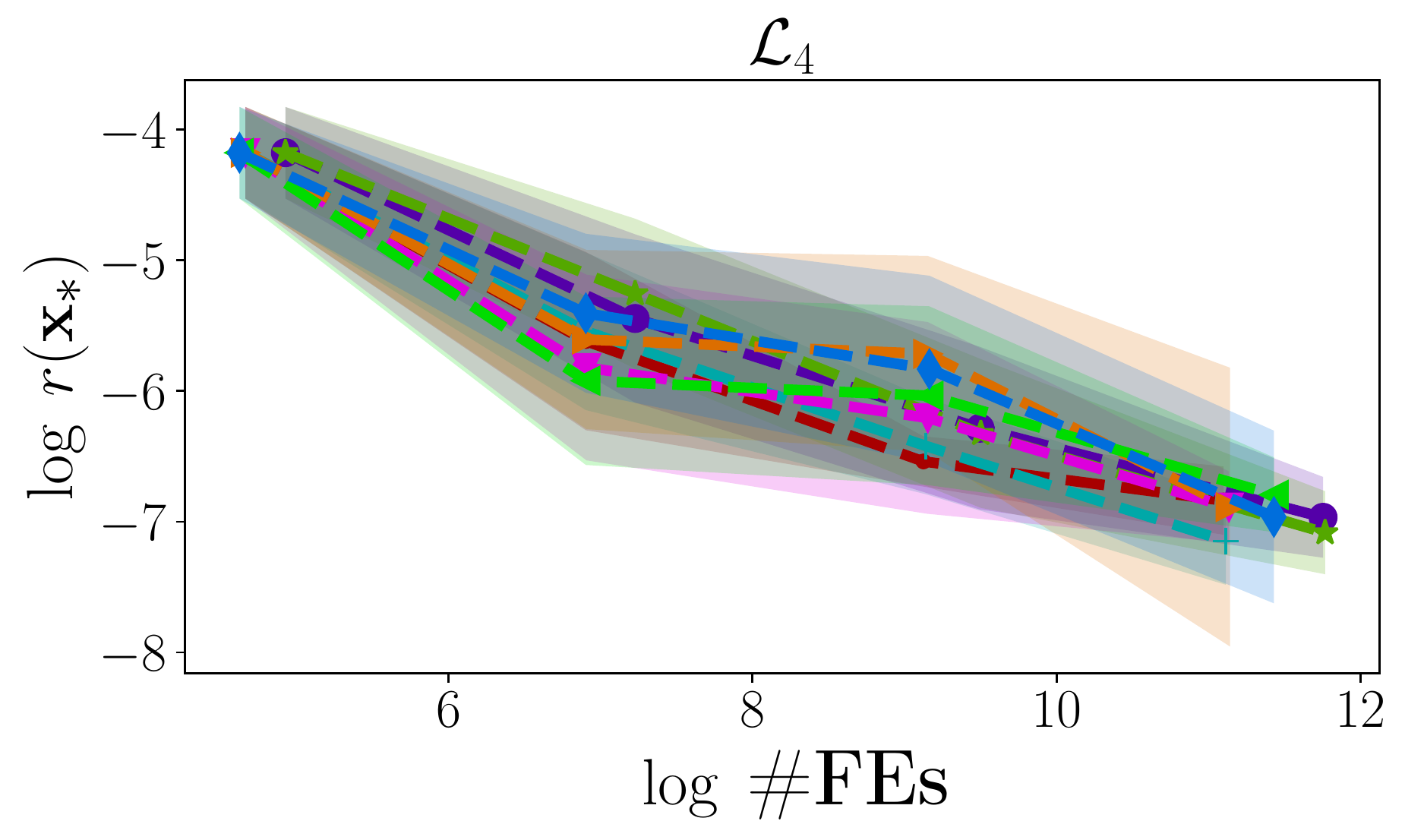} & \includegraphics[width=0.25\textwidth]{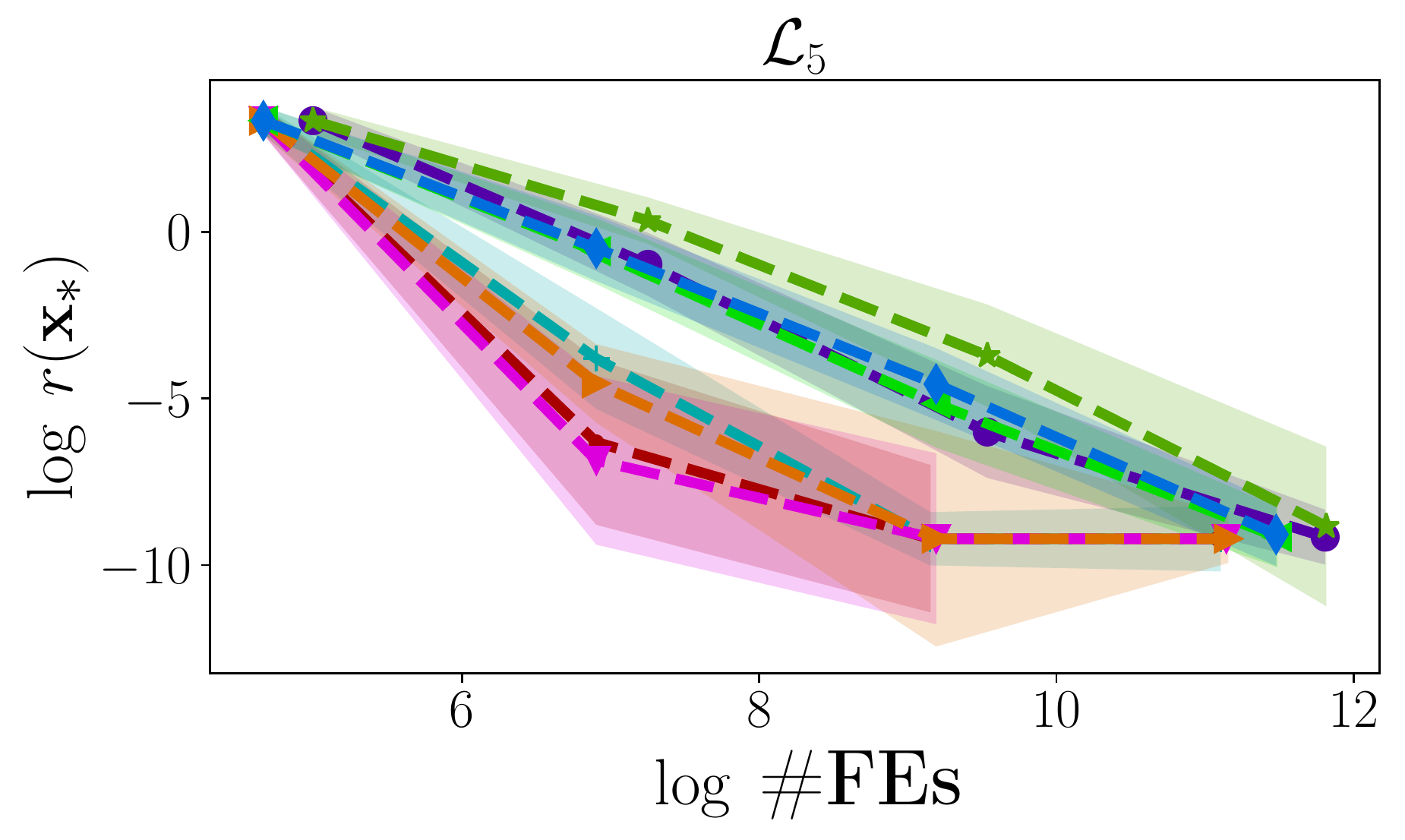} & \includegraphics[width=0.25\textwidth]{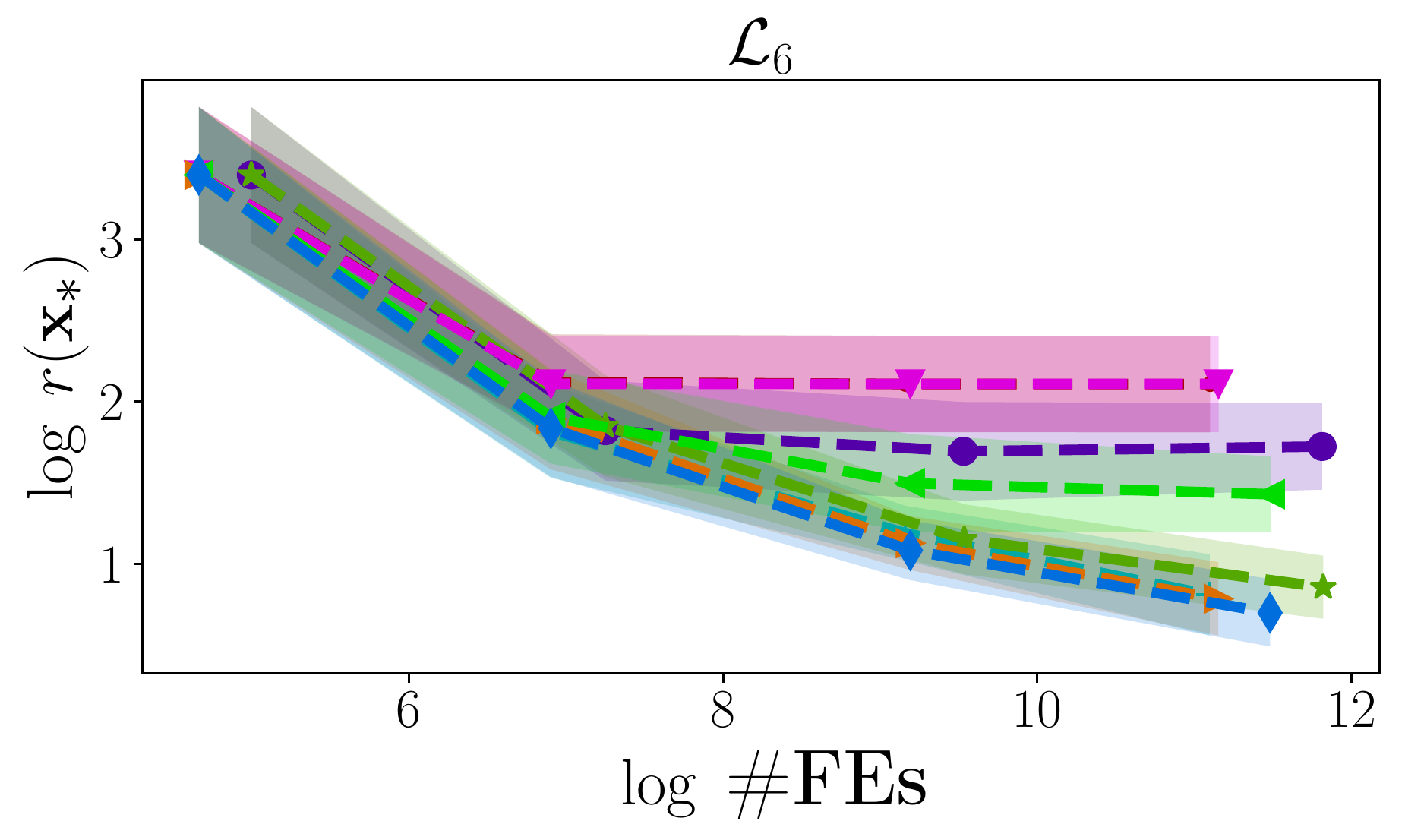}\\
	\end{tabular}
	\caption{\horse Variants. Each variant is denoted by a set of symbols as defined in Section~\ref{sec:experiments}. A: Antithetic, N: \nes, C: \cmaes, R: gradient-based restart. The markers indicate the average regret value surrounded by error bands signifying one standard deviation, obtained using $60$ independent runs.}
	\label{fig:reckless-regret}
\end{figure*}

\begin{figure*}[h!]
	\begin{tabular}{ccc}
		\includegraphics[width=0.25\textwidth]{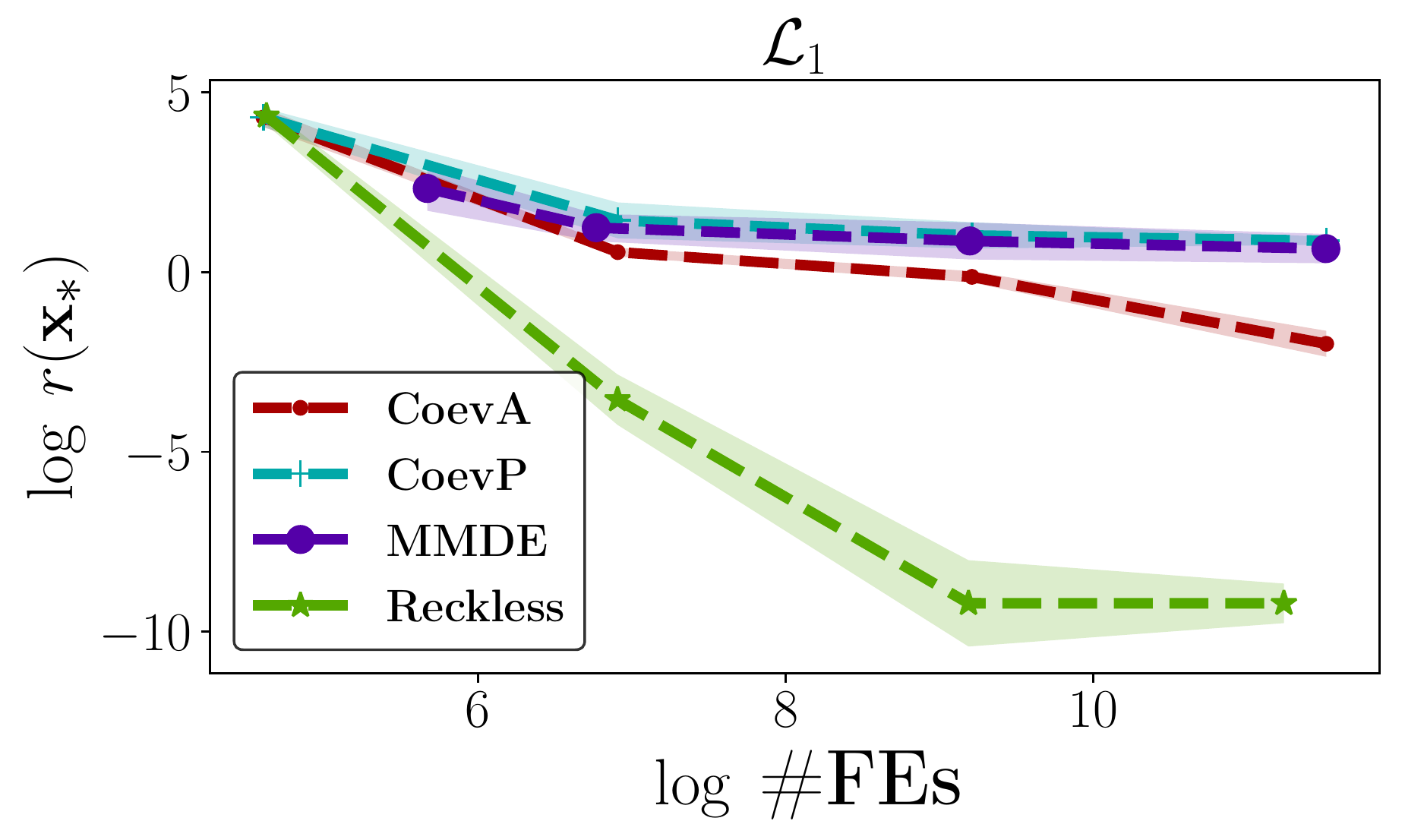} & \includegraphics[width=0.25\textwidth]{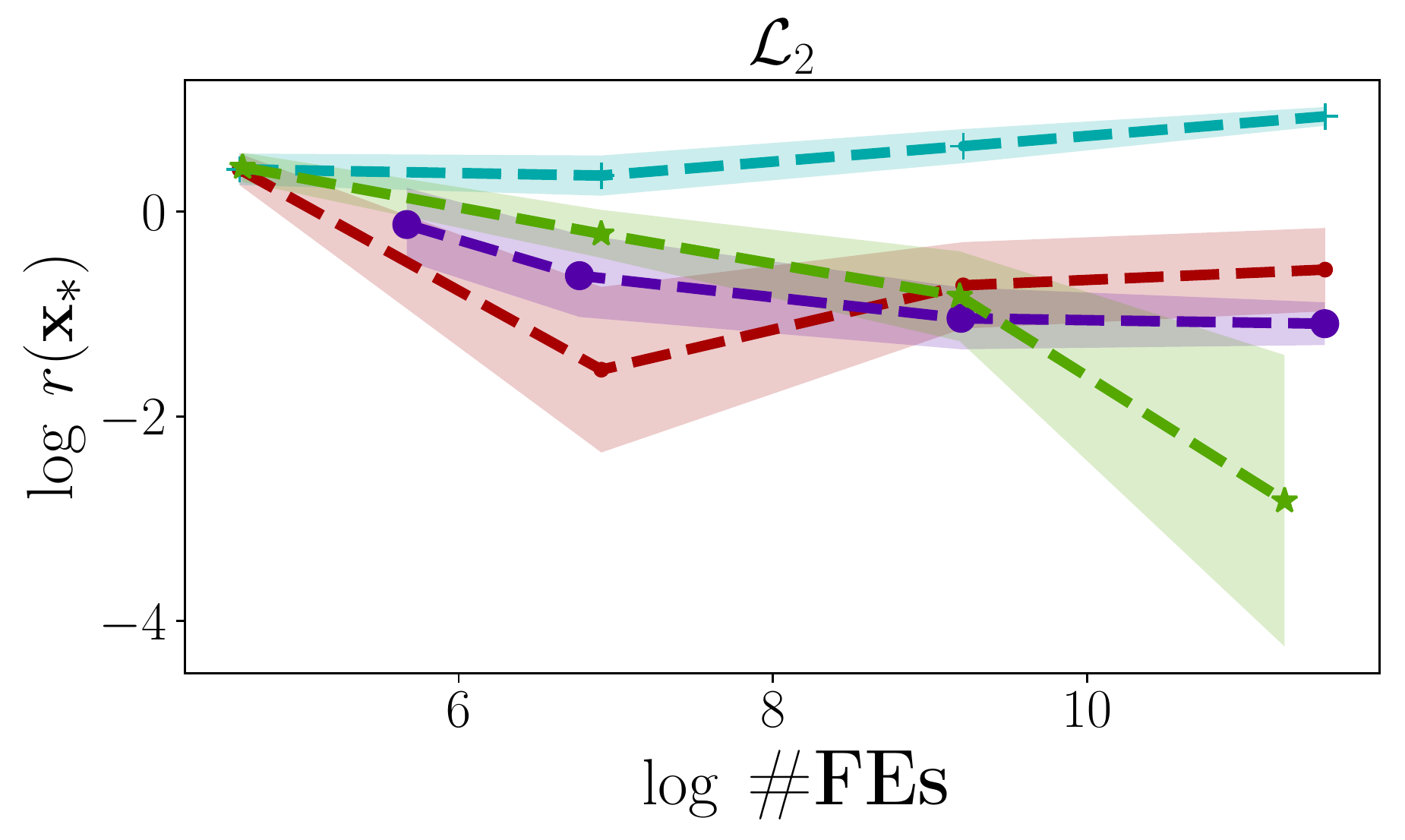} & \includegraphics[width=0.25\textwidth]{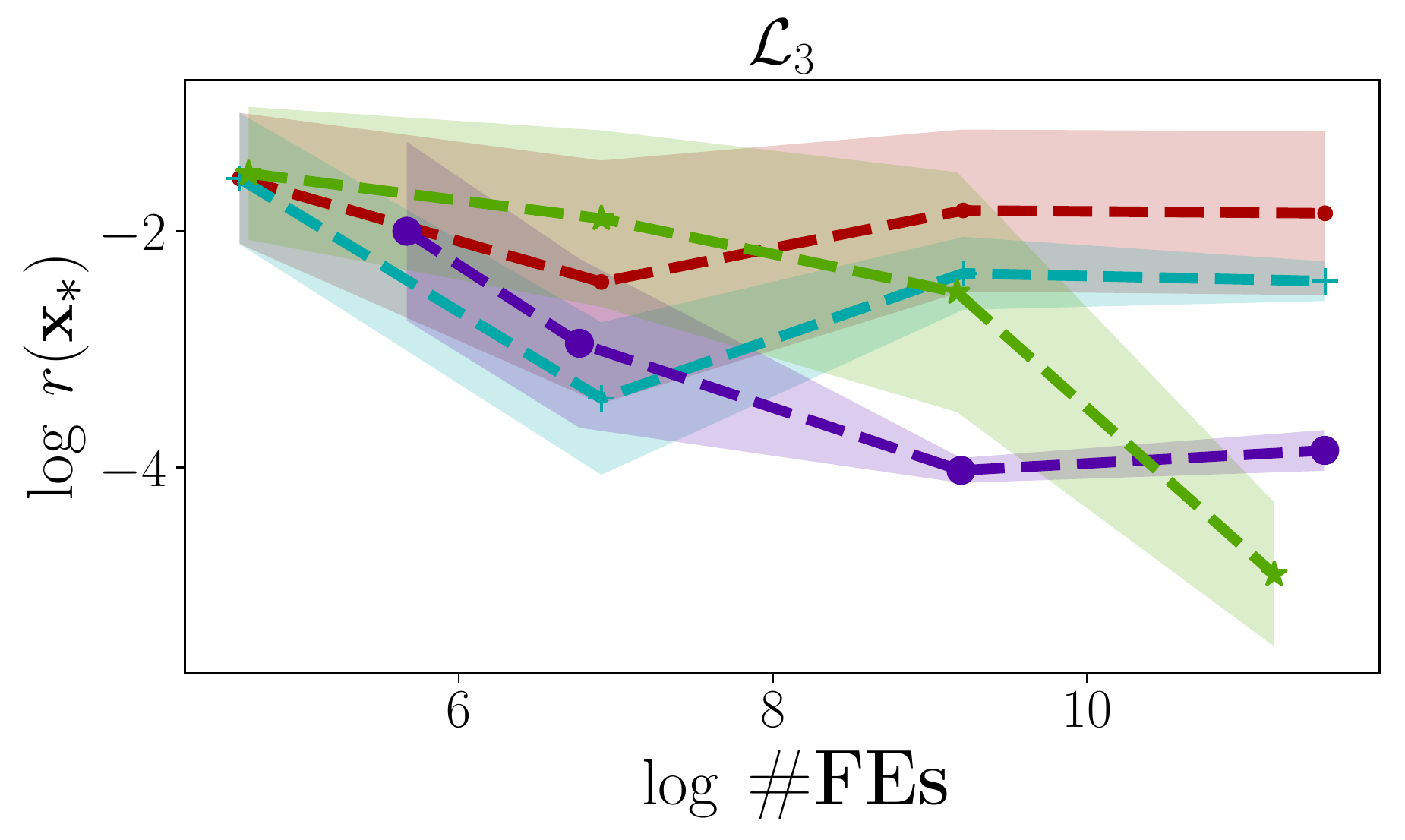}\\
		\includegraphics[width=0.25\textwidth]{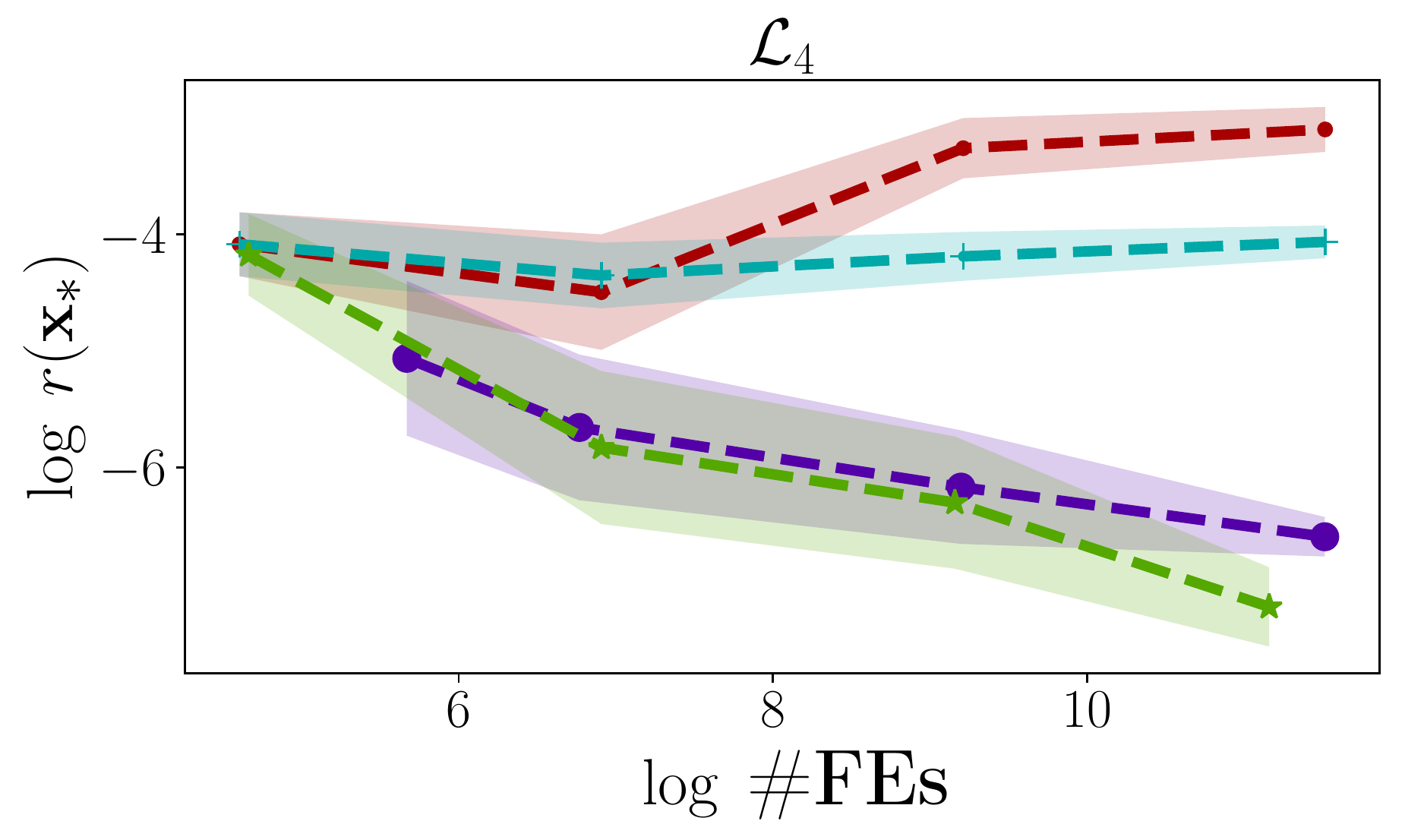} & \includegraphics[width=0.25\textwidth]{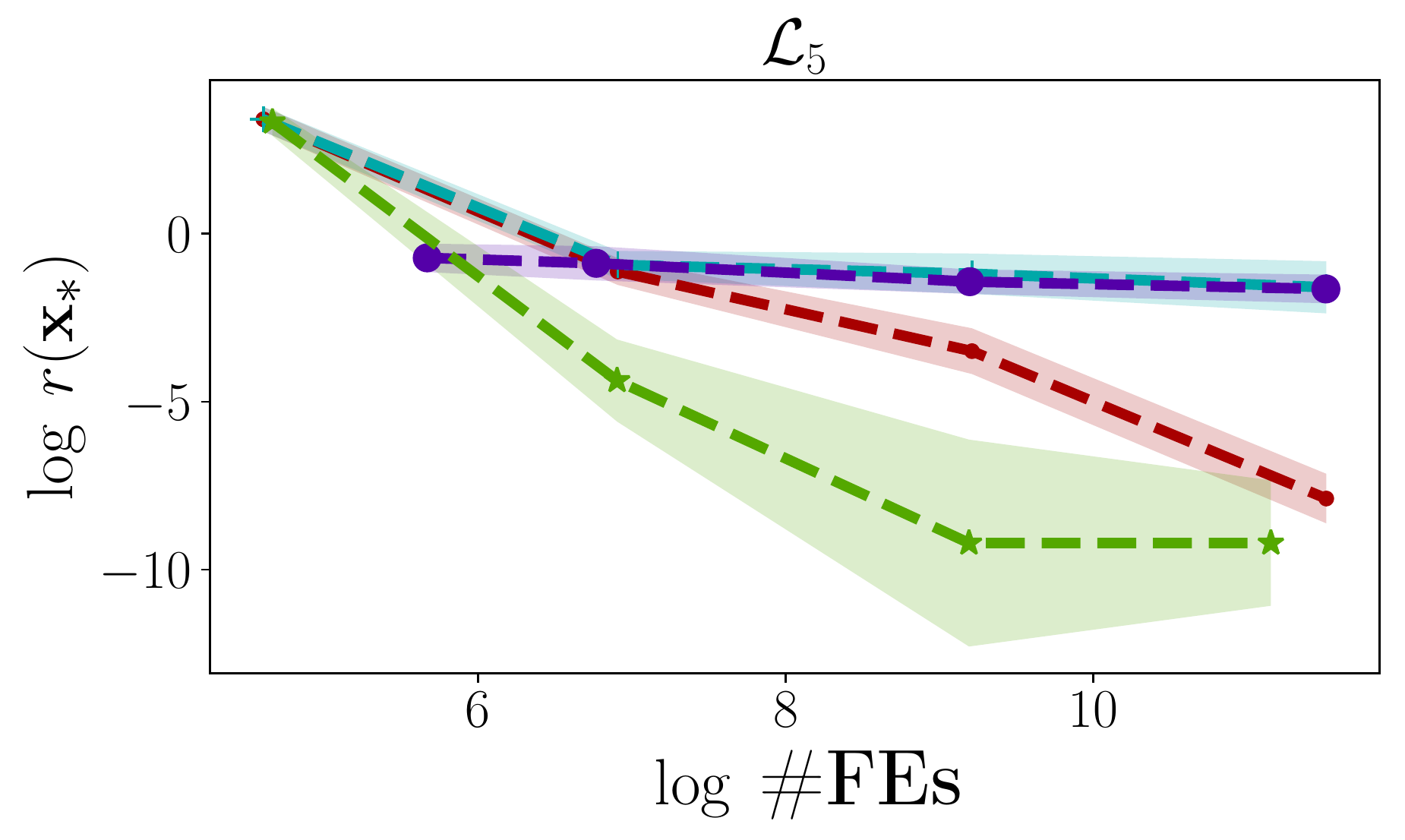} & \includegraphics[width=0.25\textwidth]{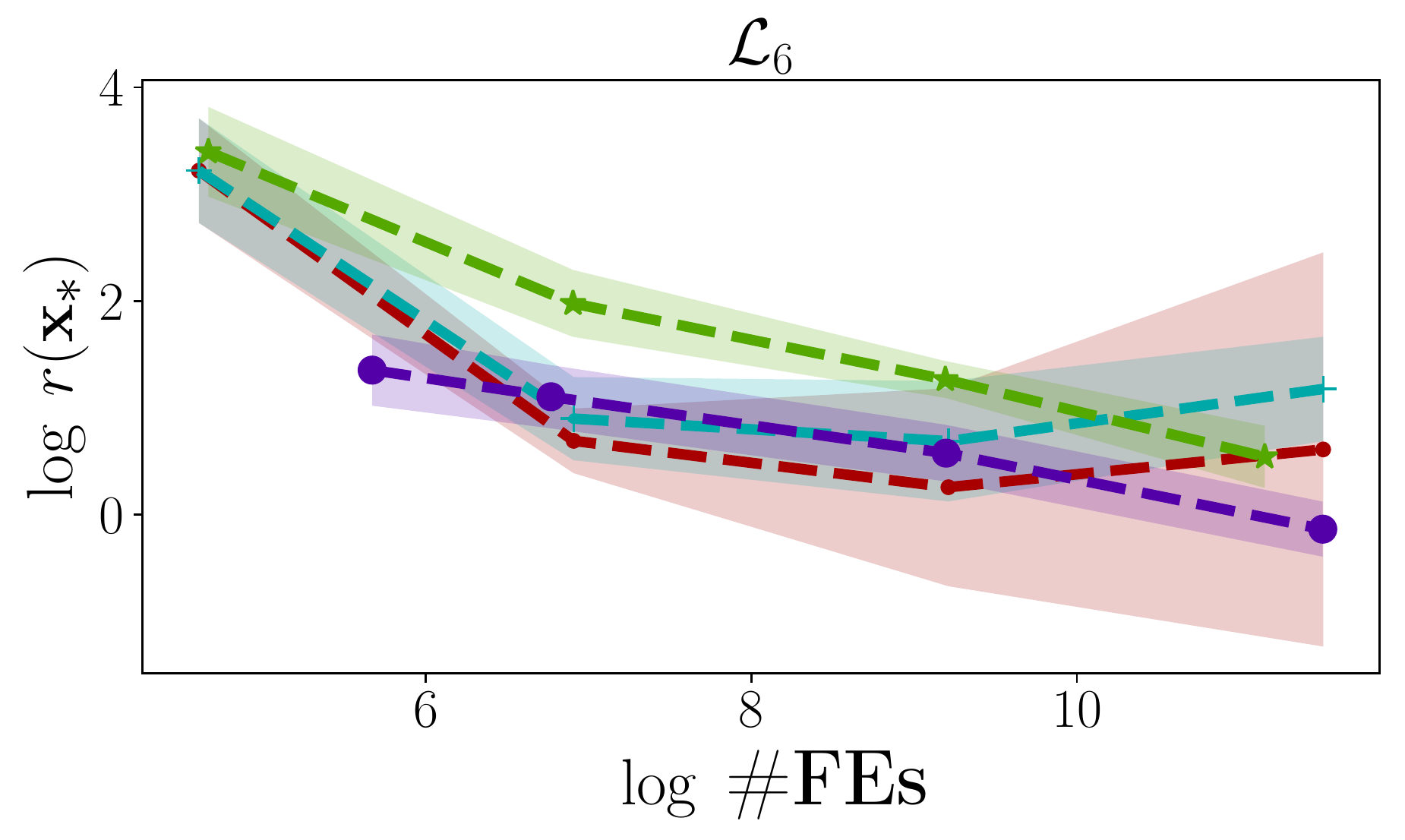}\\
	\end{tabular}
	\caption{Regret Convergence of Minimax Algorithms. The markers indicate the average regret value surrounded by error bands signifying one standard deviation, obtained using $60$ independent runs. For \horse, the \texttt{CR} variant is used with $s=0.5$.}
	\label{fig:feval-regret}
\end{figure*}

\begin{table*}[h]
	\centering
	\caption{ Description of benchmark problems. $\objfct_1$ and $\objfct_2$ were scaled  with~$n_x,n_y\in \{2, 5, 10, 15, 20, 40, 50\}$ in scalability experiments.}
	\label{tbl:benchmark-problems}
	\resizebox{0.79\textwidth}{!}{
		\begin{tabular}{p{0.3cm}p{7cm}llll}
			\toprule
			\textbf{$\objfct$} & \textbf{Definition} & $\Xspace$ & $\Yspace$ & $\vx^*$ & $\vy^*$ \\
			\midrule
			$\objfct_1$ & $\sum^{n}_{i=1} (x_i-5)^2 - (y_i-5)^2$ &
			$[0,10]^3$ & $[0,10]^3$ & $[5,5,5]$ &  $[5,5,5]$\\
			\midrule
			$\objfct_2$ & $\sum^{n}_{i=1} \min\{3-0.2x_i + 0.3y_i, 3 +0.2 x_i -0.1 y_i \}$ &
			$[0,10]^3$ & $[0,10]^3$ & $[0,0,0]$ & $[0,0,0]$\\
			\midrule
			$\objfct_3$ & $\frac{\sin(x-y)}{\sqrt{x^2+y^2}}$ &
			$(0,10]$ & $(0,10]$ & $10$ & $2.125683$\\
			\midrule
			$\objfct_4$ & $\frac{\cos\big(\sqrt{x^2+y^2}\big)}{\sqrt{x^2+y^2}+10}$ &
			$[0,10]$ & $[0,10]$ &  $7.044146$ & $0$ or $10$\\
			\midrule
			$\objfct_5$ & $100 (x_2 - x^2_1)^2 + (1-x_1)^2 - y_1(x_1 + x^2_2) - y_2(x^2_1 + x_2)$ &
			$[-0.5,0.5]\times [0,1]$ & $[0,10]^2$ &$[0.5,0.25]$ & $[0,0]$\\
			\midrule
			$\objfct_6$ & $(x_1-2)^2 + (x_2 - 1)^2 + y_1(x^2_1-x_2)+ y_2(x_1 + x_2 -2)$ &
			$[-1,3]^2$ & $[0,10]^2$ & $[1,1]$ & any $\vy \in \Yspace$\\
			\bottomrule
	\end{tabular}}
\end{table*}

\begin{table}
	\caption{\horse setup given a budget allocation $s$ for steps along the descent direction and  a finite number of function evaluations~\fes. In each of the $T$ iterations, the inner maximization and the outer minimization use $(1-s)v$ and $sv$ function evaluations, respectively. }
	\label{tbl:s-values}
	\begin{tabular}{c|cccc|cccc}
		\toprule
		$\bm{s}$ &\textbf{\fes} &  $\bm{T}$& $\bm{(1-s)v}$& $\bm{sv}$ &\textbf{\fes} &  $\bm{T}$& $\bm{(1-s)v}$& $\bm{sv}$\\
		\toprule
		$0.1$ & \multirow{5}{*}{$10^2$} & $1$ & $90$ & $10$ & \multirow{5}{*}{$10^5$} & $41$ & $2195$ & $243$ \\
		$0.2$ &  & $1$ & $80$ & $20$ &  & $38$ & $2105$ & $526$ \\
		$0.3$ &  & $1$ & $70$ & $30$  &  & $36$ & $1944$ & $833$ \\
		$0.4$ &  & $1$ & $60$ & $40$ &  & $34$ & $1764$ & $1176$ \\
		$0.5$ &  & $1$ & $50$ & $50$  &  & $32$ & $1562$ & $1562$ \\
		\bottomrule
	\end{tabular}
\end{table}

\begin{figure}[h]
	\begin{tabular}{c}
		\includegraphics[width=0.33\textwidth]{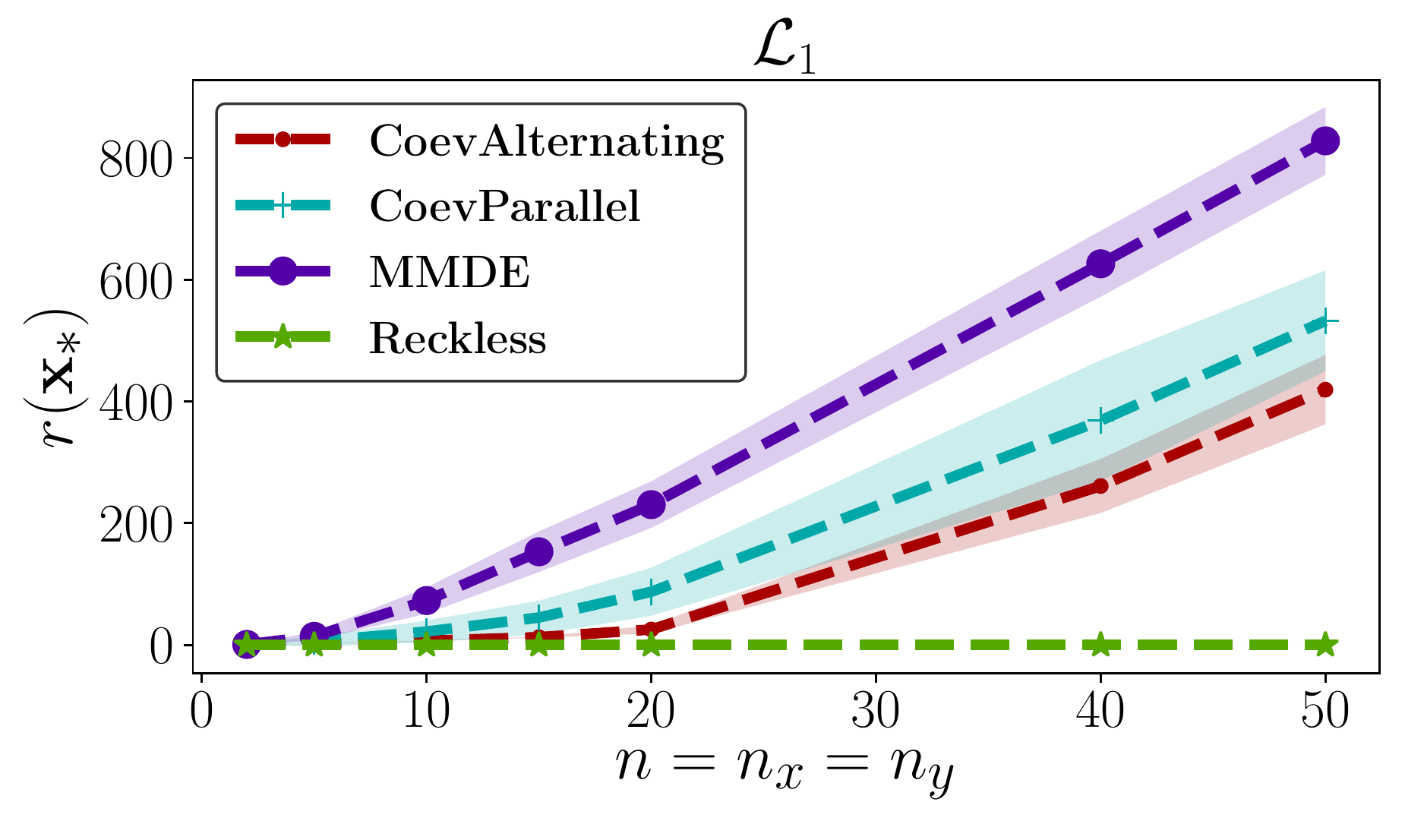} \\ \includegraphics[width=0.33\textwidth]{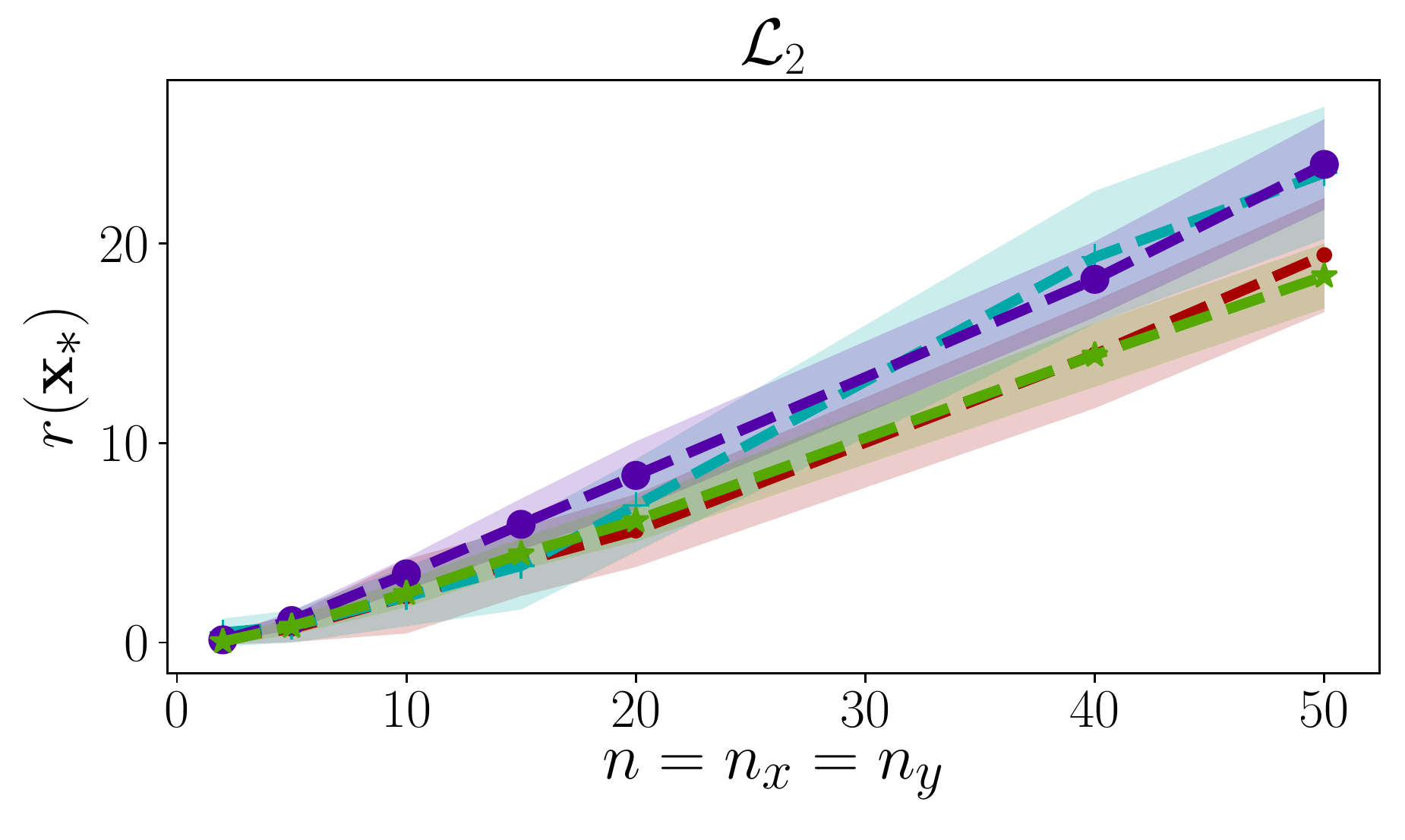} \\
	\end{tabular}
	\caption{Scalability Experiments. The compared algorithms were run on $\objfct_1$ and $\objfct_2$ with  $n_x=n_y\in \{2, 5, 10, 15, 20, 40, 50\}$ and $\fes=10^4\times n_x$. The markers indicate the average regret value surrounded by error bands signifying one standard deviation, obtained using $60$ independent runs. For \horse, the \texttt{CR} variant is used with $s=0.5$.}
	\label{fig:scale-regret}
\end{figure}

\begin{figure}[h]
	\begin{tabular}{c}
		\includegraphics[width=0.35\textwidth]{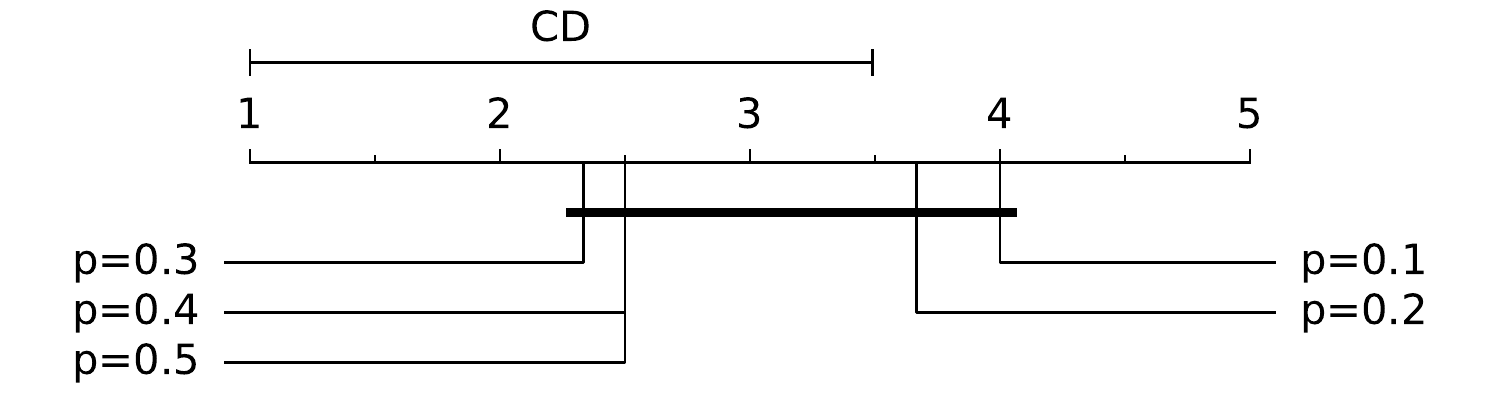} \\
		\textbf{(a) \horse steps along the descent direction}\\
		\includegraphics[width=0.35\textwidth]{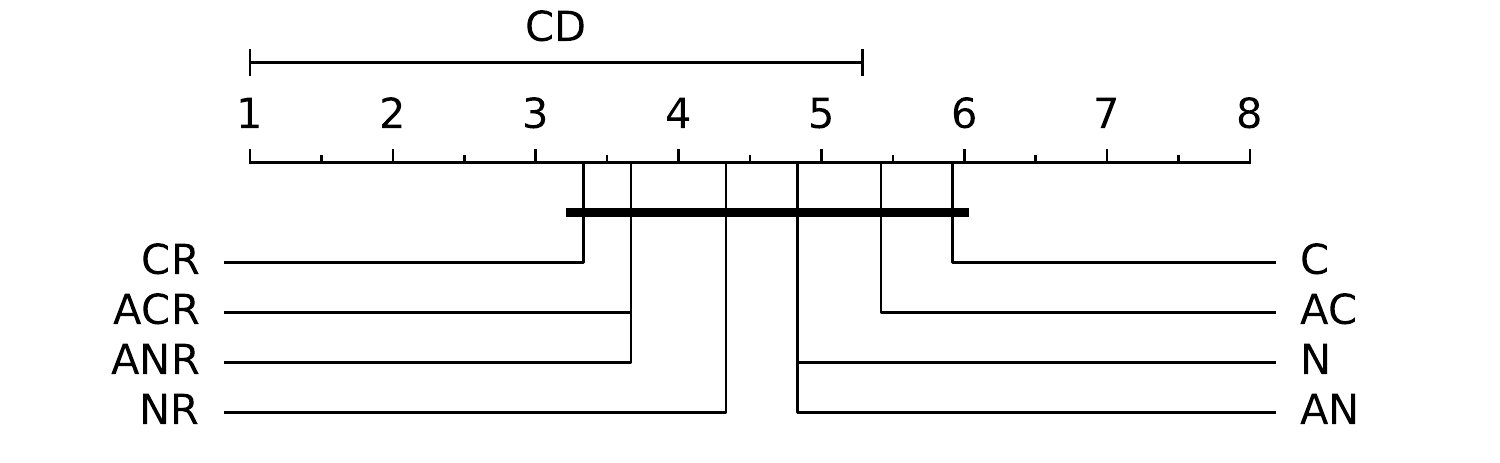} \\
		\textbf{(b) \horse variants} \\
		
		\includegraphics[width=0.35\textwidth]{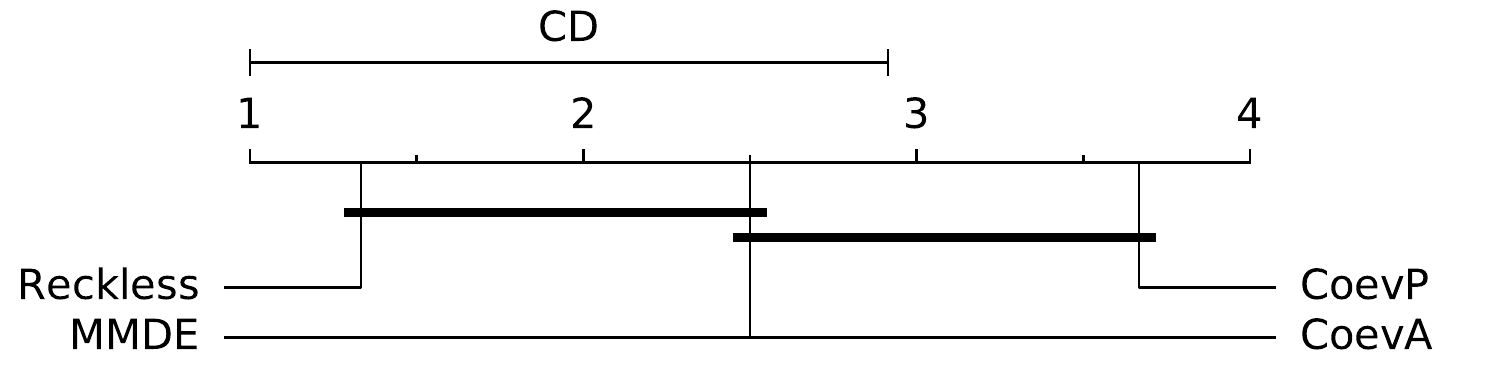} \\
		\textbf{(c) Convergence} \\ \includegraphics[width=0.35\textwidth]{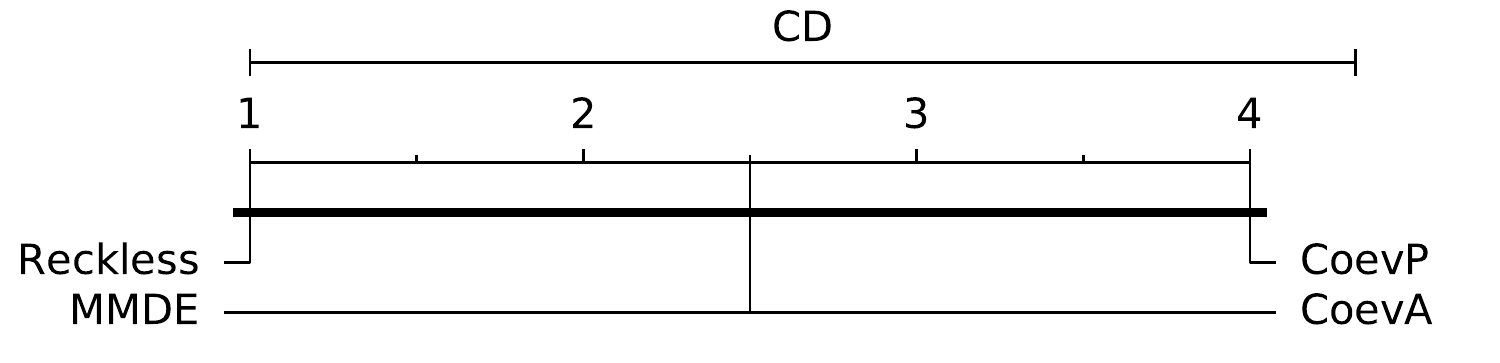} \\
		\textbf{(d) Scalability}\\
		
	\end{tabular}
	\caption{Critical Difference Plots. In our setup, a statistical significant difference was only found in (c) where \horse outperforms \coevp. The Friedman test is used when comparing the performance of two or more algorithms as reported by differing evaluators. In our setup, the evaluators are the six benchmark problems---except for (d), where the evaluators are just $\objfct_1$ and $\objfct_2$. We run a Friedman test for each of our four experiments, with the null-hypothesis of same regret. In each of the four tests, the null-hypothesis was rejected with a \textit{p-value} $< 10^{-3}$, implying that at least one of the algorithms produced different regret values. In order to confirm which algorithm performed better, we carried out a post-hoc Nemenyi test. This test produces a critical difference (CD) score when provided with average ranks of the algorithms. The difference in average ranks between any two algorithms should be greater than this CD for one to conclude that one algorithm does better than the other.}
	\label{fig:cd}
\end{figure}

\end{document}